\documentclass[11pt]{amsart}

\usepackage[utf8]{inputenc}
\usepackage[T1]{fontenc}

\usepackage{amsmath,amssymb,amsfonts,textcomp,amsthm,xifthen,graphicx,color,pgfplots}

\usepackage{enumerate}
\usepackage{enumitem}

\usepackage{pgfplots}
\usepgfplotslibrary{groupplots}
\usepgfplotslibrary{colorbrewer}
\pgfplotsset{compat=newest}
\usepgfplotslibrary{external}
\usepackage{fullpage}

\usepackage[utf8]{inputenc}

\usepackage[pdftex,
            pdftitle={A space-time finite element method for parabolic obstacle problems},
            ]{hyperref}

\newcommand{\logLogSlopeTriangle}[6]
{

    \pgfplotsextra
    {
        \pgfkeysgetvalue{/pgfplots/xmin}{\xmin}
        \pgfkeysgetvalue{/pgfplots/xmax}{\xmax}
        \pgfkeysgetvalue{/pgfplots/ymin}{\ymin}
        \pgfkeysgetvalue{/pgfplots/ymax}{\ymax}

        \pgfmathsetmacro{\xArel}{#1}
        \pgfmathsetmacro{\yArel}{#3}
        \pgfmathsetmacro{\xBrel}{#1-#2}
        \pgfmathsetmacro{\yBrel}{\yArel}
        \pgfmathsetmacro{\xCrel}{\xArel}

        \pgfmathsetmacro{\lnxB}{\xmin*(1-(#1-#2))+\xmax*(#1-#2)} 
        \pgfmathsetmacro{\lnxA}{\xmin*(1-#1)+\xmax*#1} 
        \pgfmathsetmacro{\lnyA}{\ymin*(1-#3)+\ymax*#3} 
        \pgfmathsetmacro{\lnyC}{\lnyA+#4*(\lnxA-\lnxB)}
        \pgfmathsetmacro{\yCrel}{\lnyC-\ymin)/(\ymax-\ymin)} 

        \coordinate (A) at (rel axis cs:\xArel,\yArel);
        \coordinate (B) at (rel axis cs:\xBrel,\yCrel);
        \coordinate (C) at (rel axis cs:\xCrel,\yCrel);

        \draw[#5]   (A)--
                    (B)-- 
                    (C)-- node[pos=0.5,anchor=west] {#6}
                    cycle;
    }
}

\newcommand{\logLogSlopeTriangleBelow}[6]
{

    \pgfplotsextra
    {
        \pgfkeysgetvalue{/pgfplots/xmin}{\xmin}
        \pgfkeysgetvalue{/pgfplots/xmax}{\xmax}
        \pgfkeysgetvalue{/pgfplots/ymin}{\ymin}
        \pgfkeysgetvalue{/pgfplots/ymax}{\ymax}

        \pgfmathsetmacro{\xArel}{#1}
        \pgfmathsetmacro{\yArel}{#3}
        \pgfmathsetmacro{\xBrel}{#1-#2}
        \pgfmathsetmacro{\yBrel}{\yArel}
        \pgfmathsetmacro{\xCrel}{\xArel}

        \pgfmathsetmacro{\lnxB}{\xmin*(1-(#1-#2))+\xmax*(#1-#2)} 
        \pgfmathsetmacro{\lnxA}{\xmin*(1-#1)+\xmax*#1} 
        \pgfmathsetmacro{\lnyA}{\ymin*(1-#3)+\ymax*#3} 
        \pgfmathsetmacro{\lnyC}{\lnyA+#4*(\lnxA-\lnxB)}
        \pgfmathsetmacro{\yCrel}{\lnyC-\ymin)/(\ymax-\ymin)} 

        \coordinate (A) at (rel axis cs:\xArel,\yArel);
        \coordinate (B) at (rel axis cs:\xBrel,\yCrel);
        \coordinate (C) at (rel axis cs:\xBrel,\yArel);

        \draw[#5]   (A)--
                    (B)-- node[pos=0.5,anchor=east] {#6}
                    (C)-- 
                    cycle;
    }
}

\newtheorem{theorem}{Theorem}

\newtheorem{lemma}[theorem]{Lemma}

\newtheorem{proposition}[theorem]{Proposition}
\newtheorem{remark}[theorem]{Remark}

\newcommand{\patch}{\omega}
\newcommand{\Patch}{\Omega}

\def\di{\mathrm{d}}

\def\enorm#1{|\hspace*{-.5mm}|\hspace*{-.5mm}|#1|\hspace*{-.5mm}|\hspace*{-.5mm}|}

\DeclareMathOperator*{\argmin}{arg\, min}

\newcommand{\ip}[2]{(#1\hspace*{.5mm},#2)}
\newcommand{\dual}[2]{\langle#1\hspace*{.5mm},#2\rangle}

\newcommand{\diam}{\mathrm{diam}}

\renewcommand{\div}{\operatorname{div}}

\newcommand{\divx}{\operatorname{div}_{\bx}}
\newcommand{\divst}{\operatorname{div}_{t,\bx}}
\newcommand{\gradx}{\nabla_{\bx}}

\newcommand{\Hdivxset}[1]{\boldsymbol{H}(\divx;#1)}

\newcommand{\set}[2]{\big\{#1\,:\,#2\big\}}

\newcommand{\RT}{\ensuremath{\boldsymbol{RT}}}

\newcommand{\R}{\ensuremath{\mathbb{R}}}
\newcommand{\N}{\ensuremath{\mathbb{N}}}

\newcommand{\OO}{\ensuremath{\mathcal{O}}}

\newcommand{\cV}{\ensuremath{\mathcal{V}}}
\newcommand{\cN}{\ensuremath{\mathcal{N}}}
\newcommand{\cP}{\ensuremath{\mathcal{P}}}
\newcommand{\cPtri}{\ensuremath{\mathcal{P}_{\triangle}}}
\newcommand{\cPten}{\ensuremath{\mathcal{P}_{\square}}}
\newcommand{\cU}{\ensuremath{\mathcal{U}}}
\newcommand{\cL}{\ensuremath{\mathcal{L}}}
\newcommand{\cW}{\ensuremath{\mathcal{W}}}
\newcommand{\cK}{\ensuremath{\mathcal{K}}}
\newcommand{\cKsym}{\ensuremath{\mathcal{K}}}
\newcommand{\cKsymDisc}{\ensuremath{\cKsym_{\cP}}}
\newcommand{\cKsymDiscTri}{\ensuremath{\cKsym_{\cPtri}}}
\newcommand{\cKsymDiscTen}{\ensuremath{\cKsym_{\cPten}}}

\newcommand{\bbP}{\ensuremath{\mathbb{P}}}
\newcommand{\bbS}{\ensuremath{\mathbb{S}}}

\newcommand\iop{\mathcal{I}}
\newcommand\jop{\mathcal{J}}

\newcommand{\bA}{\ensuremath{\boldsymbol{A}}}

\newcommand{\bb}{\ensuremath{\boldsymbol{b}}}

\newcommand{\blfsym}{\ensuremath{a}}

\newcommand{\rhssym}{\ensuremath{F}}

\newcommand{\bsigma}{{\boldsymbol\sigma}}

\newcommand{\btau}{{\boldsymbol\tau}}
\newcommand{\bchi}{{\boldsymbol\chi}}

\newcommand{\bx}{{\boldsymbol{x}}}
\newcommand{\by}{{\boldsymbol{y}}}

\newcommand{\bu}{\boldsymbol{u}}
\newcommand{\bv}{\boldsymbol{v}}

\begin{document}

\title{A space-time finite element method for parabolic obstacle problems}
\date{\today}

\author{Jos\'e Joaqu\'in Carvajal}
\author{Davood Damircheli}
\address{Department of Mathematics and Statistics, Mississippi State University, Mississippi State, USA}
\email{ddamircheli@lsu.edu}

\author{Thomas F\"uhrer}
\address{Facultad de Matem\'{a}ticas, Pontificia Universidad Cat\'{o}lica de Chile, Santiago, Chile}
\email{jjcarvajal@uc.cl, thfuhrer@uc.cl \emph{(corresponding author)}, francisco.fuica@uc.cl}

\author{Francisco Fuica}

\author{Michael Karkulik}
\address{Departamento de Matem\'{a}tica, Universidad T\'{e}cnica Federico Santa Mar\'{i}a, Valpara\'{i}so, Chile}
\email{michael.karkulik@usm.cl}

\thanks{{\bf Acknowledgment.}
JC was supported by ANID through Doctorado Nacional/2022-21221458.
TF was supported by ANID through FONDECYT project 1210391, FF was supported by ANID through FONDECYT project 3230126, MK was supported by ANID through FONDECYT project 1210579.
The authors thank Prof. Gregor Gantner from University of Bonn for sharing his codes.}

\begin{abstract}
  We propose and analyze a general framework for space-time finite element methods that is based on least-squares finite element methods for solving a first-order reformulation of the thick parabolic obstacle problem.
  Discretizations based on simplicial and prismatic meshes are studied and we show a priori error estimates for both.
  Convergence rates are derived for sufficiently smooth solutions.
  Reliable a posteriori bounds are provided and used to steer an adaptive algorithm. 
  Numerical experiments including a one-phase Stefan problem and an American option pricing problem are presented.
\end{abstract}
\maketitle

\section{Introduction}
\noindent
In this article we study a space-time finite element method (FEM) based on an augmented least-squares FEM (LS-FEM) for solving the parabolic obstacle problem
\begin{subequations}\label{eq:model}
\begin{alignat}{2}
  \partial_t u + \cL u &\geq f &\quad&\text{in } Q:=(0,T)\times\Omega,\label{eq:model:a} \\
  u&\geq g &\quad&\text{in } Q, \label{eq:model:b}\\
  (\partial_t u + \cL u-f)(u-g) &= 0 &\quad&\text{in } Q, \label{eq:model:c}\\
  u &= 0 &\quad&\text{on } (0,T)\times \partial\Omega, \label{eq:model:d}\\
  u(0) &= u_0 &\quad&\text{in } \Omega. \label{eq:model:e}
\end{alignat}
\end{subequations}
Here, $\Omega\subset \R^d$ with $d\in\{1,2,3\}$ is a bounded Lipschitz domain, $u_0$ is some initial datum, $f$ the load term, $g$ the obstacle, and operator $\cL$ has the form
\begin{align*}
  \cL u (t,\bx) = -\divx(\bA(t,\bx)\gradx u(t,\bx)) + \bb(t,\bx)\cdot\gradx u(t,\bx) + c(t,\bx) u(t,\bx), 
\end{align*}
$(t,\bx)\in Q = (0,T)\times \Omega$.
Under some assumptions on the coefficients of operator $\cL$ and data $f$, $g$, $u_0$, model problem~\eqref{eq:model} admits a unique (weak) solution, see, e.g.,~\cite{Brezis72,CharrierTroianello78,LionsStampacchia}. In Section~\ref{sec:analysis} below we give details on definitions and assumptions. 

Parabolic obstacle problems play an important role for optimal stopping problems, e.g., American option modelling in financial mathematics, or the one-phase Stefan problem, see~\cite{MR2256030},~\cite{MR869544}.
For a general overview on applications and mathematical formulations of obstacle problems resp. variational inequalities we refer the interested reader to~\cite{MR880369,KinderlehrerStampacchia00}, see also~\cite{GLT81} for a monograph more focused on numerical analysis of variational inequalities.

Error estimates of time-stepping methods in combination with finite elements in space have been developed at least as early as~\cite{Johnson76}, see also~\cite{Vuik90} and the more recent work~\cite{GudiMajumder19}.
A challenge that all these references have identified is the regularity of solutions to~\eqref{eq:model}. 
For instance, for $\cL u = -\Delta_{\bx} u$, and sufficiently regular data and domain $\Omega$, the solution $u$ satisfies $u\in L^2( (0,T);H^2(\Omega))\cap H^1( (0,T);H_0^1(\Omega))$, see~\cite{Brezis72} as well as~\cite[Eq.(1.3)]{Johnson76}. 
However, the global regularity $u\in H^2( (0,T); H^{-1}(\Omega))$ is in general not true. 
Even worse, the obstacle function $g$ and/or the initial datum $u_0$ are often not smooth in applications, e.g., $g\in L^2( (0,T); H^{3/2-\varepsilon}(\Omega))$ and $u_0\in H^{3/2-\varepsilon}(\Omega)$ ($\varepsilon>0$) for the American put option problem (see Section~\ref{sec:blackscholes} below).
These reasons motivate to develop and study a posteriori error estimators as in~\cite{MR2355709,MR2453213}, see also~\cite{MR4107013} for a slightly different model problem.
In general, it is expected that the free boundary moves in time. Therefore, it seems necessary to adapt the spatial mesh according to error indicators to obtain an efficient numerical solution scheme.
Mesh adaption requires refining and coarsening, where the latter, to the best of our knowledge, has not been fully analyzed for parabolic obstacle problems, see~\cite[Remark~7]{MR2453213}.

Interest in space-time finite element methods has grown over the past decades, see~\cite{MR3020957,MR4205051,MR3835623,FuehrerKarkulik21,GantnerStevenson21,DieningStevensonStorn25} to name a few recent works, and~\cite{MR1107894} for an earlier contribution.
  Motivations for studying space-time methods for solving parabolic PDEs include quasi-optimality for any given trial space and simultaneous space-time adaptivity, and, potentially, massive parallelization.
A disadvantage of space-time methods is the increased memory requirements, which, however, gets obsolete in optimal control or data assimilation problems.

Some works that consider an error analysis in space-time norms for parabolic obstacle problems include~\cite{MR3948751} and~\cite{MR3163875}.
The first one deals with a priori and a posteriori estimates for elliptic and parabolic variational inequalities involving the fractional Laplacian and time-independent obstacles. The latter is a time-stepping method that focuses on $hp$ discretizations. 
Another work that deals with parabolic variational inequalities and the fractional Laplacian is~\cite{MR3542012}.
Let us note that these works are based on separate discretization in time whereas in the work at hand we focus on fully simultaneous space-time discretizations.
Existence and uniqueness of solutions to variational inequalities involving noncoercive forms has been thoroughly studied in~\cite{GlasUrban14}. Particularly, the case of parabolic variational inequalities and discretizations of Petrov--Galerkin type are analyzed.

The aim of this article is to develop a space-time FEM for solving~\eqref{eq:model} and keep the advantageous properties noted before. As basis, we use the LS-FEMs from~\cite{FuehrerKarkulik21,GantnerStevenson21} and combine them with ideas from~\cite{LSQobstacle} which considered elliptic obstacle problems. 
  Another work that deals with LS-FEM for elliptic obstacle problems is~\cite{MR3577944}.
To obtain a suitable formulation to apply LS-FEM techniques we rewrite model~\eqref{eq:model} as first-order system with the additional variable $\lambda:= \partial_t u + \cL u -f$ which can be interpreted as a reaction force.
Then, we show that this system is equivalent to a variational inequality of the first kind where the involved bilinear form is symmetric, coercive and bounded. The discrete variational inequality itself is equivalent to the minimization problem
\begin{align*}
  \min_{(u,\bsigma,\lambda)\in \cKsymDisc} \frac12\|\partial_t u+\divx\bsigma-\lambda-f\|_{L^2(Q)}^2 + \frac12 \|\gradx u+\bsigma\|_{L^2(Q)}^2 + \frac12\|u(0)-u_0\|_{L^2(\Omega)}^2 + \frac12\dual{\lambda}{u-g},
\end{align*}
where for a simpler presentation in this introduction we have taken $\cL = -\Delta_{\bx}$, $\diam(\Omega)\leq 1$ and $\cKsymDisc$ denotes a non-empty, closed and convex set which includes the conditions $u\geq g$, $\lambda \geq 0$ in a discrete sense.
The functional in the minimization problem is of least-squares type augmented by the (duality) term $\dual{\lambda}{u-g}$ which implements the complementarity condition for~\eqref{eq:model}.
We prove that our proposed method is quasi-optimal in the sense of the C\'ea--Falk lemma for any choice of discrete trial space and set $\cKsymDisc$.
Furthermore, we consider discretizations based on simplicial meshes and prismatic meshes which are extensions of the FEM spaces analyzed in~\cite{FuehrerKarkulik21} and~\cite{GantnerStevenson24,StevensonStorn23}, respectively.
Particularly, we give details on convergence rates for prismatic meshes by extending some results from~\cite{StevensonStorn23} and combine them with positivity preserving (quasi-)interpolation operators.
An a-posteriori error estimate is provided that consists of least-squares terms plus terms that measure the violation of the discretized complementarity condition and the penetration of the discrete solution with the obstacle.

\subsection{Outline}
The remainder of this document is organized as follows: 
In Section~\ref{sec:analysis} we revisit well-posedness of problem~\eqref{eq:model} and recast it into a first-order system. 
We then introduce our LS-FEM based method, prove its well-posedness and define discretization spaces.
Section~\eqref{sec:aprioriposteriori} deals with a priori and a posteriori estimates. For the a priori analysis on tensor meshes presented in Section~\ref{sec:apriori} we rely on quasi-interpolation operators introduced in Section~\ref{sec:quasi_inter_tensor}.
A posteriori estimators are studied in Section~\ref{sec:aposteriori}.
The work is concluded by numerical experiments (Section~\ref{sec:numerics}) where we study different setups including the Stefan problem and the American put option pricing.

\section{First-order formulation and least-squares discretization}\label{sec:analysis}

\subsection{Sobolev and Bochner--Sobolev spaces}
Throughout this work, we use common notations for Lebesgue and Sobolev spaces.
In particular, we set $H^{k}(\Omega) := W^{k,2}(\Omega)$ with $k\in\{1,2,3\}$. 
We denote by $H_0^1(\Omega)$ the closure of $C_0^\infty(\Omega)$ under $\|\nabla v\|_{L^{2}(\Omega)}$, whereas $H^{-1}(\Omega)$ denotes the dual space of $H_0^1(\Omega)$. 
The duality on $H^{-1}(\Omega)\times H_0^1(\Omega)$ is written as
\begin{align*}
  \dual{\phi}{v} \quad\text{for } \phi\in H^{-1}(\Omega), \, v \in H_0^1(\Omega).
\end{align*}
Note that $\dual{\phi}{v} = \ip{\phi}v_{L^{2}(\Omega)}$ for all $v\in H_0^1(\Omega)$ and $\phi\in L^2(\Omega)$. 
The dual space $H^{-1}(\Omega)$ is equipped with norm
\begin{align*}
  \|\phi\|_{H^{-1}(\Omega)}:= \sup_{0\neq v\in H_0^{1}(\Omega)} \frac{\dual{\phi}v}{\|\nabla v\|_{L^{2}(\Omega)}}, \quad \phi \in H^{-1}(\Omega).
\end{align*}

Given a Hilbert space $H$ with norm $\|\cdot\|_{H}$ and $v : (0,T)\to H$ (strongly measurable), we set  
\begin{align*}
\|v\|_{L^{2}((0,T);H)}^{2} := \int_{0}^{T}\|v(t)\|_{H}^{2}\mathrm{d}t, 
\qquad
\|v\|_{H^{k}((0,T);H)}^{2} := \|v\|_{L^{2}((0,T);H)}^{2} + \sum_{j=1}^k\|\partial_{t}^{j}v\|_{L^{2}((0,T);H)}^{2},
\end{align*}
where $\partial_{t}^{j}$ denotes the $j$-th (weak) derivative with respect to the time variable.
We consider the following Bochner spaces
\begin{align*}
L^2( (0,T); H) & :=\{v: (0,T) \to H  : \, \|v\|_{L^{2}((0,T);H)} < \infty\}, \\
H^k( (0,T); H) & :=\{v: (0,T) \to H : \, \|v\|_{H^{k}((0,T);H)} < \infty\}.
\end{align*}
Note that $L^{2}((0,T);L^{2}(\Omega))$ and $L^{2}(Q)$ are isometrically isomorph and we identify these spaces for the remainder of this article.
For the analysis of the variational formulations we use the spaces
\begin{alignat*}{2}
  \cV &= L^2( (0,T);H_0^1(\Omega)), 
  &\cV^*=\,& L^2( (0,T);H^{-1}(\Omega)), \\
  \cW &= \cV\cap H^1( (0,T);H^{-1}(\Omega)), \qquad
  &\widetilde\cW =\,& L^2( (0,T);H^1(\Omega)) \cap H^1( (0,T);H^{-1}(\Omega)).
\end{alignat*}

We recall the following well-known integration by parts result, see, e.g.~\cite[Section~5.9, Theorem~3]{Evans98}.
\begin{theorem}[embedding, integration by parts]\label{thm:Bochner_space_prop}
Space $\cW$ is continuously embedded in the space $C^0([0,T];L^{2}(\Omega))$.
Moreover, the integration by parts formula 
\begin{align*}
\int_{0}^{T}\dual{\partial_{t}v(t)}{w(t)}\mathrm{d}t = - \int_{0}^{T}\dual{\partial_{t}w(t)}{v(t)}\mathrm{d}t + (v(T),w(T))_{L^{2}(\Omega)} - (v(0),w(0))_{L^{2}(\Omega)}
\end{align*}
holds true for all $v, w \in \cW$.
\end{theorem}

For $v \in L^2(Q)$ we write $v\geq 0$ if the inequality holds a.e. in $Q$. 
For elements $\mu\in \cV^*$ we define $\mu\geq 0$ by $\dual{\mu}{v}_{\cV^*\times \cV}\geq 0$ for all $v\in\cV$ with $v\geq 0$.
Furthermore, $v\geq w$ means $v-w\geq 0$.
The positive and negative part of a function $v\in L^2(Q)$ are denoted with $v^+ = \max\{v,0\}$ and $v^- = -\min\{v,0\}$, respectively. Here, $\max\{\cdot,\cdot\}$ and $\min\{\cdot,\cdot\}$ are defined pointwise. 
We note that $w^+ = \max\{0,w\}\in\cV$ for all $w\in L^2( (0,T);H^1(\Omega))$ with $w|_{(0,T)\times\partial\Omega}\leq 0$.

Note that any $\mu\in \cV^*$ with $\mu\geq 0$ induces a (unique) non-negative Radon measure (also denoted by $\mu$) with
\begin{align*}
  \int_Q v \,\mathrm{d}\mu = \dual{\mu}{v}_{\cV^*\times \cV} \quad\forall v \in \cV\cap C^0(\overline{Q}).
\end{align*}
With a slight abuse of notation, we will write $\dual{\mu}{g}_{\cV^*\times \cV} = \dual{g}{\mu}_{\cV\times \cV^*}$ instead of $\int_Qg \,\di\mu$.

We close this section by noting that $\divx\colon L^2(Q)^d\to \cV^*$ is a linear operator with
\begin{align*}
  \|\divx\bsigma\|_{\cV^*} = \sup_{v\in \cV\setminus\{0\}} \frac{|\dual{\divx\bsigma}{v}_{\cV^*\times \cV}|}{\|v\|_{\cV}}
  = \sup_{v\in \cV\setminus\{0\}} \frac{|\ip{\bsigma}{\gradx v}_{L^2(Q)}|}{\|v\|_{\cV}} \leq \|\bsigma\|_{L^2(Q)} 
\end{align*}
for all $\bsigma\in L^2(Q)^d$.
Let $0<c_F\leq \diam(\Omega)$ denote Friedrich's constant, i.e., the smallest constant $c_F>0$ such that
\begin{align}\label{eq:friedrich}
  \|v\|_{L^2(Q)}\leq c_F \|\gradx v\|_{L^2(Q)} \quad\text{for all } v\in\cV.
\end{align}
Consequently, $\|\phi\|_{L^2( (0,T);H^{-1}(\Omega))}\leq c_F \|\phi\|_{L^2(Q)}$ for all $\phi\in L^2(Q)$.

\subsection{Well-posedness of canonic variational formulation}
We consider a canonic variational space-time formulation of problem~\eqref{eq:model} and state its well-posedness as given in~\cite{CharrierTroianello78}, see Theorem~\ref{thm:weakform} below.
To that end we impose some rather mild assumptions on operator $\cL$. 
Suppose that $\bA^\top = \bA\in L^\infty(Q)^{d\times d}$, $\bb\in L^\infty(Q)^d$, $c\in L^\infty(Q)$ with $L^\infty(Q)\ni-\tfrac12\divx\bb + c\geq 0$ such that there exists $\alpha>0$ with
\begin{alignat}{2}\label{eq:alpha}
\by\cdot \bA\by &\geq \alpha |\by|^2 &\quad& \forall \by\in \R^d, \quad \text{a.e. in } Q. 
\end{alignat}
The last inequality and $-\tfrac12\divx\bb + c\geq 0$ imply that 
\begin{align}\label{eq:coerc_L}
  \dual{\cL v}{v}_{\cV^*\times \cV} \geq \alpha \|v\|_{\cV}^2 \quad\forall v\in \cV.
\end{align}

Let $f\in \cV^*$, $u_0\in L^2(\Omega)$ be given. 
For the obstacle function $g\in \widetilde\cW\cap C^0(\overline Q)$ we assume that
\begin{subequations}\label{eq:assumptionsObs}
\begin{align}
  g|_{(0,T)\times \partial\Omega} \leq 0, 
  \quad
  g(0)\leq u_0.
\end{align}
In addition, we define $\psi := \partial_t g + \cL g -f \in \cV^*$ and suppose that the decomposition $\psi = \psi^+-\psi^-$ holds with
\begin{align}
  0\leq \psi^\pm \in \cV^*.
\end{align}
\end{subequations}
Hence, with the non-empty convex and closed (in $\cV$) set 
\begin{align*}
  \cK_\cV = \set{v\in \cV}{v\geq g},
\end{align*}
the variational formulation of~\eqref{eq:model} reads as follows:
\begin{align}\label{eq:weakform}
\begin{split}
  \text{find } u\in \cW\cap \cK_\cV \text{ such that }
  \dual{\partial_t u + \cL u}{v-u}_{\cV^*\times \cV} &\geq \dual{f}{v-u}_{\cV^*\times \cV} \quad\forall v\in \cK_\cV, \\
  u(0) &= u_0.
\end{split}
\end{align}

\begin{theorem}[{\cite[Theorem~1 and Remark~2]{CharrierTroianello78}}]\label{thm:weakform}
  Problem~\eqref{eq:weakform} admits a unique solution.
\qed
\end{theorem}

\begin{remark}
  Note that Theorem~\ref{thm:weakform} is valid if $\cL$ satisfies the more general estimate
  \begin{align*}
    \dual{\cL v}{v}_{\cV^*\times \cV} \geq \beta\|v\|_{\cV}^2 - \gamma \|v\|_{L^2(Q)}^2 \quad\forall v\in \cV, 
  \end{align*}
  for some constants $\beta>0$, $\gamma\geq 0$, cf.~\cite{CharrierTroianello78}.
\qed
\end{remark}

\subsection{First-order system and variational formulation}\label{sec:fo}
In this section we derive an equivalent first-order system of~\eqref{eq:model}, state a variational formulation thereof, and prove its well-posedness. Given $(v,\btau)\in L^{2}(Q)^{1+d}$ define the (distributional) time-space divergence operator $\divst(v,\btau) := \partial_{t}v + \div_{\bx}\btau$.
Let $u\in \cV$ denote the weak solution of~\eqref{eq:weakform} with $f\in L^2(Q)$.
Set $\bsigma := -\bA\gradx u \in L^2(Q)^d$ and
\begin{align*}
  \lambda:=\partial_t u + \cL u -f = \divst(u,\bsigma)+\bb\cdot\gradx u + cu -f\in \cV^*.
\end{align*}
Then, the triplet $(u,\bsigma,\lambda)$ satisfies the first-order system given by
\begin{subequations}\label{eq:fo}
\begin{alignat}{2}
  \divst(u,\bsigma)+\bb\cdot\gradx u+cu-\lambda &= f &\quad&\text{in } Q, \label{eq:fo:a}\\
  \bsigma + \bA\gradx u &= 0 &\quad&\text{in } Q, \label{eq:fo:b} \\
  u\geq g, \quad \lambda\geq 0, \quad \lambda(u-g) &= 0 &\quad&\text{in } Q,\label{eq:fo:c} \\
  u &= 0 &\quad&\text{on } (0,T)\times \partial\Omega, \\
  u(0) &= u_0  &\quad&\text{in } \Omega \label{eq:fo:e}.
\end{alignat}
\end{subequations}
Based on this system we derive a novel variational formulation. 
To do so, we first need to find an appropriate functional analytic setting.
We set
\begin{align*}
  \cU &= \set{(v,\btau,\mu)\in \cV\times L^2(Q)^d \times \cV^*}{\divst(v,\btau)-\mu\in L^2(Q)}.
\end{align*}
This definition is motivated by the fact that if $(u,\bsigma,\lambda)\in \cV\times L^2(Q)^d\times \cV^*$ satisfies~\eqref{eq:fo:a}, then $\divst(u,\bsigma)-\lambda = f-\bb\cdot\gradx u -cu \in L^2(Q)$.
We stress that $\cU$ is a Hilbert space when equipped with the graph norm, denoted by $\|\cdot\|_\cU$.
However, the latter norm is too strong and we need to work with a weaker norm.
First, observe that the definition of $\cU$ implies that if $(u,\bsigma,\lambda)\in \cU$, then, $\divst(u,\bsigma)-\lambda\in L^2(Q)$.
Therefore, $\partial_t u = -\divx\bsigma+\lambda + w \in \cV^*$ for some $w\in L^2(Q)$, thus, $u\in\cW$.
Define $|\cdot|_\cU\colon \cU\to \R_{\geq 0}$ for $(u,\bsigma,\lambda)\in\cU$ by
\begin{align*}
  |(u,\bsigma,\lambda)|_{\cU}^2 &:= \|u\|_{\cV}^2 + \|u(0)\|_{L^2(\Omega)}^2 +  \|u(T)\|_{L^2(\Omega)}^2 + \|\bsigma\|_{L^2(Q)}^2 + \|\divst(u,\bsigma)-\lambda\|_{L^2(Q)}^2.
\end{align*}
To see that the latter is a norm, we note that homogeneity and the triangle inequality are trivial to verify. 
For definiteness, let $(u,\bsigma,\lambda)\in\cU$ be given and assume that $|(u,\bsigma,\lambda)|_\cU = 0$. 
Then, $u=0$, $\bsigma = 0$, and $\divst(u,\bsigma)-\lambda = 0$. 
Therefore,  $\lambda = \divst(u,\bsigma)=0$. 
Norm $|\cdot|_\cU$ can be bounded from above by the canonic norm $\|\cdot\|_{\cU}$, i.e., 
\begin{align*}
  |(u,\bsigma,\lambda)|_{\cU}^2 &= \|u\|_{\cV}^2 + \|u(0)\|_{L^2(\Omega)}^2 +  \|u(T)\|_{L^2(\Omega)}^2 + \|\bsigma\|_{L^2(Q)}^2 + \|\divst(u,\bsigma)-\lambda\|_{L^2(Q)}^2
  \\
  &\lesssim \|u\|_{\cV}^2 + \|\partial_t u\|_{\cV^*}^2 + \|\bsigma\|_{L^2(Q)}^2 + \|\divst(u,\bsigma)-\lambda\|_{L^2(Q)}^2 \\
  &\lesssim \|u\|_{\cV}^2 + \|\divst(u,\bsigma)-\lambda\|_{\cV^*}^2 +\|\lambda\|_{\cV^*}^2 + \|\divx\bsigma\|_{\cV^*}^2 
  \\
  &\qquad + \|\bsigma\|_{L^2(Q)}^2 + \|\divst(u,\bsigma)-\lambda\|_{L^2(Q)}^2
  \\
  &\lesssim \|u\|_{\cV}^2 + \|\bsigma\|_{L^2(Q)}^2 + \|\lambda\|_{\cV^*}^2 + \|\divst(u,\bsigma)-\lambda\|_{L^2(Q)}^2
 = \|(u,\bsigma,\lambda)\|_{\cU}^2
\end{align*}
for all $(u,\bsigma,\lambda)\in \cU$. Here, we have used the embedding $\cW\hookrightarrow C^0([0,T];L^2(\Omega))$ from Theorem \ref{thm:Bochner_space_prop}, boundedness of $\divx\colon L^2(Q)^d\to \cV^*$, and $\|\cdot\|_{\cV^*}\lesssim \|\cdot\|_{L^2(Q)}$.
The fact that $|\cdot|_\cU$ and $\|\cdot\|_\cU$ are not equivalent can be seen from the following simple construction: Let $w\in H_0^1(\Omega)\setminus\{0\}$ and define for each $n\in\N$, $(u_n,\bsigma_n,\lambda_n)\in \cU$ by $u_n(t,\bx) = n^{-1/2}(t/T)^n w(\bx)$ for all $(t,\bx)\in Q$, $\bsigma_n = 0$, and $\lambda_n = \partial_t u_n$. A straightforward computation yields that there exists $C=C(w)>0$ with
\begin{align*}
  |(u_n,\bsigma_n,\lambda_n)|_\cU \to 0, \qquad 
  \|(u_n,\bsigma_n,\lambda_n)\|_\cU \to C \quad\text{for } n\to\infty.
\end{align*}
Particularly, $|(u,\bsigma,\lambda)|_{\cU}$ does not control the time derivative of $u$ in $\cV^*$.

Let $0<c_F\leq \diam(\Omega)$ denote Friedrich's constant~\eqref{eq:friedrich} and recall that $\alpha>0$ is the lower bound on the eigenvalues of $\bA$, see~\eqref{eq:alpha}. 
Fix some $\Lambda\geq c_F^2$, e.g., $\Lambda = \diam(\Omega)^2$ and 
define the bilinear form $\blfsym\colon \cU\times \cU\to \R$, 
\begin{align*}
  \blfsym(\bu,\bv) &:=  \alpha^{-1} \Lambda \ip{\divst(u,\bsigma)+\bb\cdot\gradx u + cu-\lambda}{\divst(v,\btau)+\bb\cdot\gradx v + cv-\mu}_{L^2(Q)} 
  \\
  &\qquad\qquad + \ip{\bA^{-1/2}\bsigma+\bA^{1/2}\gradx u}{\bA^{-1/2}\btau+\bA^{1/2}\gradx v}_{L^2(Q)}  \\
  &\qquad\qquad + \ip{u(0)}{v(0)}_{L^2(\Omega)} + \frac12\dual{\lambda}{v}_{\cV^*\times \cV} + \frac12\dual{\mu}{u}_{\cV^*\times \cV},
\end{align*}
for all $\bu=(u,\bsigma,\lambda)$, $\bv = (v,\btau,\mu)\in \cU$.

Recall that $g\in C^0(\overline Q)$ with $g|_{(0,T)\times\partial\Omega}\leq 0$, thus, $\int_Q g\,\di \mu \leq \|\max\{g,0\}\|_\cV\|\mu\|_{\cV^*}<\infty$ for $\mu\in \cV^*$ with $\mu\geq 0$. 
For any $\bv=(v,\btau,\mu)\in\cU$ with $\mu\geq 0$ set
\begin{align*}
  \rhssym(\bv) &:= \alpha^{-1}\Lambda \ip{f}{\divst(v,\btau)+\bb\cdot\gradx v + cv-\mu}_{L^2(Q)} + \ip{u_0}{v(0)}_{L^2(\Omega)} + \frac12\int_Q g\,\di \mu.
\end{align*}

We derive a variational formulation of the first-order reformulation~\eqref{eq:fo} of the obstacle problem.
To that end we introduce the set
\begin{align*}
  \cKsym &= \set{(v,\btau,\mu)\in \cU}{v\in \cK_\cV,\,\mu\geq 0},
\end{align*}
which is non-empty and convex.
Suppose that $(u,\bsigma,\lambda)\in \cKsym$ is a solution of~\eqref{eq:fo}. 
Let $(v,\btau,\mu)\in\cKsym$ be given. By~\eqref{eq:fo:c} it follows that
$\dual{\lambda}{u-g}_{\cV^*\times\cV}=0$, $\dual{\mu}{u-g}_{\cV^*\times\cV}\geq 0$, and $\dual{\lambda}{v-g}_{\cV^*\times\cV}\geq 0$. 
Combining the latter observations yields
\begin{align}\label{eq:obsinequalities}
  \frac12\dual{\mu-\lambda}{u-g}_{\cV^*\times\cV} \geq 0, \qquad
  \frac12\dual\lambda{v-u}_{\cV^*\times \cV} = \frac12\dual{\lambda}{v-g}_{\cV^*\times\cV} \geq 0.
\end{align}
Then,~\eqref{eq:fo:a}-\eqref{eq:fo:b} and~\eqref{eq:fo:e} imply for any $(w,\bchi,\nu)\in \cU$ that
\begin{align}\label{eq:fovarres}
  \begin{split}
  &\frac{\Lambda}{\alpha}\ip{\divst(u,\bsigma)+\bb\cdot\gradx u+cu-\lambda}{\divst(w,\bchi)+\bb\cdot\gradx w+cw-\nu}_{L^2(Q)}
  \\
  &\qquad +\ip{\bA^{1/2}\gradx u+\bA^{-1/2}\bsigma}{\bA^{1/2}\gradx w + \bA^{-1/2}\bchi}_{L^2(Q)}
  + \ip{u(0)}{w(0)}_{L^2(\Omega)} \\
  &\qquad\qquad = \frac{\Lambda}{\alpha}\ip{f}{\divst(w,\bchi)+\bb\cdot\gradx w+cw-\nu}_{L^2(Q)} + \ip{u_0}{w(0)}_{L^2(\Omega)}.
\end{split}
\end{align}
Adding~\eqref{eq:fovarres} with $(w,\bchi,\nu) = (v-u,\btau-\bsigma,\mu-\lambda)$ and~\eqref{eq:obsinequalities} shows that any solution $\bu=(u,\bsigma,\lambda)\in\cKsym$ of~\eqref{eq:fo} satisfies the following variational inequality:
\begin{align}\label{eq:vi:sym}
  \text{Find } \bu\in \cKsym: \quad
  \blfsym(\bu,\bv-\bu) \geq \rhssym(\bv-\bu) \quad\forall \bv\in \cKsym.
\end{align}

In what follows we show well-posedness of variational inequality~\eqref{eq:vi:sym}, but first we prove boundedness and coercivity of the bilinear form $\blfsym(\cdot,\cdot)$.

\begin{lemma}\label{lem:propblfsym}
  Bilinear form $\blfsym(\cdot,\cdot)$ is symmetric and bounded with respect to $|\cdot|_\cU$, i.e.,
  \begin{align*}
    \blfsym(\bu,\bv) = \blfsym(\bv,\bu), \qquad |\blfsym(\bu,\bv)|\lesssim |\bu|_\cU |\bv|_\cU \quad\forall \bu,\bv\in \cU.
  \end{align*}
  Furthermore, $\blfsym(\cdot,\cdot)$ is $|\cdot|_\cU$-coercive, i.e.,
  \begin{align*}
    |\bu|_\cU^2 \lesssim \blfsym(\bu,\bu) \quad\forall \bu\in \cU.
  \end{align*}
\end{lemma}
\begin{proof}
  Symmetry can be seen from the definition. Next, we verify boundedness. 
  To that end let $\bu = (u,\bsigma,\lambda)$, $\bv = (v,\btau,\mu)\in\cU$ be given and observe that
  \begin{align*}
    |\blfsym(\bu,\bv)| &\lesssim \|\divst(u,\bsigma)+\bb\cdot\gradx u + cu-\lambda\|_{L^2(Q)}
                             \\ &\qquad\qquad \times  \|\divst(v,\btau)+\bb\cdot\gradx v + cv-\mu\|_{L^2(Q)}
                             \\
                       & \qquad + \|\bA^{1/2}\gradx u+\bA^{-1/2}\bsigma\|_{L^2(Q)}\|\bA^{1/2}\gradx v+\bA^{-1/2}\btau\|_{L^2(Q)} 
                       \\
                     & \qquad + \|u(0)\|_{L^2(\Omega)}\|v(0)\|_{L^2(\Omega)} 
                     + \frac12|\dual{\mu}{u}_{\cV^*\times \cV} + \dual{\lambda}{v}_{\cV^*\times \cV}|
                     \\
                     & \lesssim |\bu|_{\cU}|\bv|_{\cU} 
                     + \frac12|\dual{\mu}{u}_{\cV^*\times \cV} + \dual{\lambda}{v}_{\cV^*\times \cV}|.
  \end{align*}
  Here, we applied the Cauchy--Schwarz inequality and triangle inequality.
  To tackle the last term on the right-hand side, note that
  \begin{align*}
    \dual{\mu}{u}_{\cV^*\times \cV} + \dual{\lambda}{v}_{\cV^*\times \cV}
   = & \, \dual{\divst(v,\btau)-\mu}{-u}_{\cV^*\times \cV} + \dual{\divst(u,\bsigma)-\lambda}{-v}_{\cV^*\times \cV}
  \\
  &+ \dual{\partial_t v}{u}_{\cV^*\times\cV} + \dual{\partial_t u}{v}_{\cV^*\times\cV}
  + \dual{\divx\btau}u_{\cV^*\times\cV} + \dual{\divx\bsigma}v_{\cV^*\times\cV}.
  \end{align*}
  Then, using boundedness of $\divx\colon L^2(Q)^d\to \cV^*$, $\|\cdot\|_{\cV^*}\lesssim \|\cdot\|_{L^2(Q)}$, and integration in time we get that
  \begin{align*}
    |\dual{\mu}{u}_{\cV^*\times \cV} + \dual{\lambda}{v}_{\cV^*\times \cV}| &\lesssim |\bu|_{\cU}|\bv|_{\cU} + |\dual{\partial_t v}{u}_{\cV^*\times\cV} + \dual{\partial_t u}{v}_{\cV^*\times\cV}|
    \\
    &= |\bu|_{\cU}|\bv|_{\cU} + |\ip{u(T)}{v(T)}_{L^2(\Omega)}-\ip{u(0)}{v(0)}_{L^2(\Omega)}| \lesssim |\bu|_{\cU}|\bv|_{\cU}.
  \end{align*}
  Putting all estimates together we conclude boundedness of $\blfsym(\cdot,\cdot)$.

It remains to show coercivity. To that end, let $\bv = (v,\btau,\mu)\in \cU$ be given and note that
  \begin{align*}
    \blfsym(\bv,\bv) &= \frac{\Lambda}{\alpha} \|\divst(v,\btau)+\bb\cdot\gradx v + cv - \mu\|_{L^2(Q)}^2 + \|v(0)\|_{L^2(\Omega)}^2 
     \\
     &\qquad\qquad + \|\bA^{-1/2}\btau + \bA^{1/2}\gradx v\|_{L^2(Q)}^2+ \dual{\mu}v_{\cV^*\times \cV}.
  \end{align*}
  Then, Young's inequality and integration by parts prove
  \begin{align*}
  \|\bA^{-1/2}\btau+\bA^{1/2}\gradx v\|_{L^2(Q)}^2 &= \|\bA^{-1/2}\btau\|_{L^2(Q)}^2 + 2\ip{\bA^{-1/2}\btau}{\bA^{1/2}\gradx v}_{L^2(Q)} + \|\bA^{1/2}\gradx v\|_{L^2(Q)}^2 \\
&= \|\bA^{-1/2}\btau\|_{L^2(Q)}^2 + \ip{\bA^{-1/2}\btau}{\bA^{1/2}\gradx v}_{L^2(Q)} + \ip{\btau}{\gradx v}_{L^2(Q)}\\
&\qquad + \|\bA^{1/2}\gradx v\|_{L^2(Q)}^2 \\
  &\geq \frac14 \|\bA^{-1/2}\btau\|_{L^2(Q)}^2 + \frac23 \|\bA^{1/2}\gradx v\|_{L^2(Q)}^2 + \dual{\divx\btau}{-v}_{\cV^*\times \cV}.
  \end{align*}
  Since $-\tfrac12\divx\bb + c \geq 0$, we have that $\ip{\bb\cdot\gradx v + cv}v_{L^2(Q)}\geq 0$. Using this we obtain that
  \begin{align*}
    \dual{\divx\btau-\mu}{-v}_{\cV^*\times \cV} 
    &= \dual{\divx\btau + \bb\cdot\gradx v+cv-\mu}{-v}_{\cV^*\times \cV} + \ip{\bb\cdot\gradx v+cv}{v}_{L^2(Q)} 
    \\
    &\geq \dual{\divx\btau + \bb\cdot\gradx v+cv-\mu}{-v}_{\cV^*\times \cV} 
    \\
    &= \dual{\divst(v,\btau)+\bb\cdot\gradx v+cv-\mu}{-v}_{\cV^*\times \cV} + \dual{\partial_t v}{v}_{\cV^*\times \cV}
    \\
    &\geq -\|\divst (v,\btau) + \bb\cdot\gradx v+cv-\mu\|_{L^2(Q)}\|v\|_{L^2(Q)} + \dual{\partial_t v}{v}_{\cV^*\times \cV}.
  \end{align*}
  With Friedrich's inequality, $c_F\leq \Lambda^{1/2}$ and $\alpha^{1/2}\|v\|_{\cV}\leq\|\bA^{1/2}\nabla_\bx v\|_{L^2(Q)}$ we further estimate
  \begin{align*}
    \dual{\divx\btau-\mu}{-v}_{\cV^*\times \cV} &\geq -\|\divst (v,\btau)+ \bb\cdot\gradx v+cv-\mu\|_{L^2(Q)}\|v\|_{L^2(Q)} + \dual{\partial_t v}{v}_{\cV^*\times \cV} \\
    &\geq -\|\divst (v,\btau) + \bb\cdot\gradx v+cv-\mu\|_{L^2(Q)}\alpha^{-1/2}c_F\|\bA^{1/2}\gradx v\|_{L^2(Q)} 
    \\
    &\qquad\qquad+ \dual{\partial_t v}{v}_{\cV^*\times \cV}
    \\
    &\geq -\frac{\Lambda}{2\alpha}\|\divst (v,\btau) + \bb\cdot\gradx v+cv-\mu\|_{L^2(Q)}^2 - \frac12\|\bA^{1/2}\gradx v\|_{L^2(Q)}^2 
    \\
    &\qquad\qquad+ \dual{\partial_t v}{v}_{\cV^*\times \cV}.
  \end{align*}
  Combining all estimates so far and $\dual{\partial_t v}{v}_{\cV^*\times \cV} = \tfrac12(\|v(T)\|_{L^2(\Omega)}^2-\|v(0)\|_{L^2(\Omega)}^2)$ gives
  \begin{align}\label{eq:proof:coercivity1}
  \begin{split}
    a(\bv,\bv) &\geq \frac{\Lambda}{2\alpha} \|\divst(v,\btau)+\bb\cdot\gradx v + cv - \mu\|_{L^2(Q)}^2
    + \frac14 \|\bA^{-1/2}\btau\|_{L^2(Q)}^2 + \frac16 \|\bA^{1/2}\gradx v\|_{L^2(Q)}^2
    \\
    &\qquad\qquad + \frac12\|v(0)\|_{L^2(\Omega)}^2 + \frac12\|v(T)\|_{L^2(\Omega)}^2.
  \end{split}
  \end{align}
  Employing the triangle inequality and $\|\bb\cdot\gradx v\|_{L^2(Q)} + \|c v\|_{L^2(Q)} \lesssim \|v\|_{\cV} \eqsim \|\bA^{1/2}\gradx v\|_{L^2(Q)}$ proves that 
  \begin{align*}
    |(v,\btau,\mu)|_\cU^2 &\lesssim \|\divst(v,\btau)+\bb\cdot\gradx v + cv - \mu\|_{L^2(Q)}^2 + \|\bb\cdot\gradx v\|_{L^2(Q)}^2 + \|c v\|_{L^2(Q)}^2\\
    &\qquad +\|\btau\|_{L^{2}(Q)}^{2} + \|v\|_\cV^2 + \|v(0)\|_{L^2(\Omega)}^2 +\|v(T)\|_{L^2(\Omega)}^2
    \\
    &\lesssim \|\divst(v,\btau)+\bb\cdot\gradx v + cv - \mu\|_{L^2(Q)}^2  +\|\btau\|_{L^{2}(Q)}^{2} + \|v\|_\cV^2+ \|v(0)\|_{L^2(\Omega)}^2 +\|v(T)\|_{L^2(\Omega)}^2. 
  \end{align*}
  Hence, the right-hand side of~\eqref{eq:proof:coercivity1} is bounded below by a multiple of $|(v,\btau,\mu)|_{\cU}^2$. This finishes the proof.
\end{proof}

\begin{theorem}\label{thm:exis_uniq_1st_ord_refo}
  For any $f\in L^2(Q)$, $u_0\in L^2(\Omega)$, and  $g\in \widetilde\cW\cap C^0(\overline Q)$ satisfying~\eqref{eq:assumptionsObs}, variational inequality~\eqref{eq:vi:sym} admits a unique solution.
\end{theorem}
\begin{proof}
  Let $u\in\cW\cap\cK_{\cV}$ denote the solution of~\eqref{eq:weakform}. 
  Set $\bu = (u,-\bA\gradx u,\partial_t u+\cL u-f)\in \cKsym$. Then, $\bu$ is a solution of~\eqref{eq:fo} and by construction, $\bu$ also satisfies~\eqref{eq:vi:sym}. This proves existence of a solution.

  To see uniqueness, let $\bu_1$, $\bu_2 \in \cKsym$ denote solutions of~\eqref{eq:vi:sym} which implies
  \begin{align*}
    \blfsym(\bu_1,\bu_2-\bu_1) \geq \rhssym(\bu_2-\bu_1) \quad\text{and}\quad
    \blfsym(\bu_2,\bu_1-\bu_2) \geq \rhssym(\bu_1-\bu_2).
  \end{align*}
  Adding both inequalities yields $\blfsym(\bu_1-\bu_2,\bu_2-\bu_1)\geq 0$. By coerciveness (Lemma~\ref{lem:propblfsym}) we conclude that $0\geq \blfsym(\bu_1-\bu_2,\bu_1-\bu_2) \gtrsim |\bu_1-\bu_2|_\cU^2$, hence, $\bu_1=\bu_2$.
\end{proof}

\subsection{Numerical method}
Let $\cP$ denote a partition of $Q$. 
Suppose that $\cU_\cP$ is a finite-dimensional subspace of  $\cU \cap \left( L^\infty(Q)\times L^\infty(Q)^d\times L^\infty(Q)\right)$ and that $\cKsymDisc\subset \cU_\cP$ is a non-empty convex subset but not necessarily a subset of $\cKsym$. 
Below we discuss two possibilities to define $\cP$, $\cU_\cP$, and $\cKsymDisc$. The first one is a simple discretization using Lagrange finite elements, whereas the second one is the \emph{Gantner--Stevenson element}~\cite{GantnerStevenson24} defined over prismatic meshes. It is designed to achieve better convergence speeds for less regular solutions compared to the simpler Lagrange finite element.

For either case, the discrete scheme reads: Find $\bu_\cP\in \cKsymDisc$ such that
\begin{align}\label{eq:vi:sym:disc}
  \blfsym(\bu_\cP,\bv-\bu_\cP) \geq \rhssym(\bv-\bu_\cP) \quad\forall \bv\in \cKsymDisc.
\end{align}

\begin{theorem}
  For any $f\in L^2(Q)$,  $u_0\in L^2(\Omega)$, and $g\in \widetilde\cW\cap C^0(\overline Q)$ satisfying~\eqref{eq:assumptionsObs} problem~\eqref{eq:vi:sym:disc} admits a unique solution. 
\end{theorem}
\begin{proof}
  Note that the definition of $\rhssym$ given above can be extended to functions in $\cU_\cP$ and $\rhssym|_{\cU_\cP}$ is bounded due the assumption $\mu\in L^\infty(Q)$ for $(v,\btau,\mu)\in \cU_\cP$ and $g\in C^0(\overline Q)$ and $\dim(\cU_\cP)<\infty$.
  Furthermore, $\blfsym$ is bounded and coercive on $\cU_\cP\subset\cU$ by Lemma~\ref{lem:propblfsym}.
  The Lions--Stampacchia Theorem, cf.~\cite{LionsStampacchia,KinderlehrerStampacchia00}, then implies unique solvability of~\eqref{eq:vi:sym:disc}.
\end{proof}

Since $\blfsym(\cdot,\cdot)$ is a symmetric, bounded and coercive bilinear form on $\cU_{\cP}$, variational inequality~\eqref{eq:vi:sym:disc} can equivalently be written as a convex minimization problem, i.e., 
\begin{align}\label{eq:minprobdisc}
\begin{split}
  \bu_\cP = \argmin_{(v,\btau,\mu)\in\cKsymDisc} &\frac\Lambda{2\alpha} \|\divst(v,\btau)+\bb\cdot\gradx v+cv-\mu-f\|_{L^2(Q)}^2 \\
  & + \frac12\|\bA^{1/2}\gradx v + \bA^{-1/2}\btau\|_{L^2(Q)}^2 + \frac12\|v(0)-u_0\|_{L^2(\Omega)}^2 + \frac12\dual{\mu}{v-g}_{\cV^*\times\cV}.
\end{split}
\end{align}
We stress that the functional in the minimization problem is an \emph{augmented} least-squares functional. The first three terms measure the residual of~\eqref{eq:fo:a},\eqref{eq:fo:b} and~\eqref{eq:fo:e} in squared $L^2$ norms and the last term implements the complementarity condition of~\eqref{eq:fo:c}.

\subsubsection{Simplicial meshes}\label{sec:lagrangefem}
Let $\cPtri$ denote a conforming partition of $Q$ into simplices. 
The vertices of $\cPtri$ are denoted by $\cN_{\cPtri}$.

With $\bbP_k(S)$ we denote the space of polynomials of degree $\leq k\in\N_0$ with domain $S$ and with $\bbP_k(\cPtri)$ we denote $\cPtri$-piecewise polynomials of degree $\leq k$.
Define 
\begin{align*}
  \bbS_k(\cPtri) & := \bbP_k(\cPtri)\cap H^1(Q), \qquad
  \bbS_{k,0}(\cPtri) := \bbS_k(\cPtri)\cap \cV, \qquad (k\in\N),\\
  \cU_{\cPtri} & := \bbS_{1,0}(\cPtri) \times \bbS_1(\cPtri)^{d} \times\bbP_0(\cPtri), 
  \\
  \cKsymDiscTri & := \set{(v,\btau,\mu)\in \cU_{\cPtri}}{v(z)\geq g(z), \,\forall z\in\cN_{\cPtri}, \quad \mu\geq 0}.
\end{align*}
It is straightforward to verify that $\cU_{\cPtri}\subset \cU$. 
However, note that $\cKsymDiscTri$ is not necessarily a subset of $\cKsym$ unless the obstacle $g$ is in the space $\bbS_1(\cPtri)$.

\subsubsection{Prismatic meshes}\label{sec:gsfem}
We recall the definition of the finite element constructed in~\cite{GantnerStevenson24}.
Let $\cPten$ denote a partition of $Q$ into open, nonoverlapping prisms $P = J\times K$ where $J\subseteq (0,T)$ and $K\subseteq \Omega$ is a simplex.
Let $\bbP_{\ell,k}(P) = \bbP_\ell(J)\otimes \bbP_k(K)$, $(\ell,k)\in\N_0^2$ and
\begin{align*}
  \widetilde\bbP_{\ell,k}(P) = \bbP_\ell(J)\otimes \RT_k(K),
\end{align*}
where $\RT_k(K) = \bbP_k(K)+\bx \bbP_k(K)\subset \bbP_{k+1}(K)$ denotes the Raviart--Thomas space.
Note that for $d=1$ this space reduces to $\bbP_{k+1}(K)$.
In~\cite[Section~3]{GantnerStevenson24} space 
\begin{align*}
  \bbS_{\ell,k}(\cPten) = \set{(v,\btau)\in \cV\times L^2(Q)^d}{\divst(v,\btau)\in L^2(Q), \, (v,\btau)|_P \in \bbP_{\ell+1,k}(P)\times \widetilde\bbP_{\ell,k}(P) \, \forall P\in\cPten}
\end{align*}
is introduced. 

The partition $\cPten$ can either be conforming, i.e., two distinct elements $P,P'$ that touch each other share a facet, or non-conforming to some degree.
The latter is necessary to allow for local mesh-refinements. In what follows we use $|J|=\diam(J)$ and $|K|$ is the measure of the simplex $K$.
We follow~\cite[Section~3.2]{GantnerStevenson24} and say that $\cPten$ is non-conforming if for two distinct elements $P=J\times K$, $P'=J'\times K'$ that touch each other we have that $|J|\eqsim |J'|$, $|K|\eqsim |K'|$ and either $P$, $P'$ share a facet or $\overline{P}\cap\overline{P'}$ is a facet $F$ of $P$ (or $F'$ of $P'$) and a proper subset of a facet $F'$ of $P'$ (or $F$ of $P$). 
Suppose that it is a subset of $P$, then $P'$ ($F'$) is called the primary element (facet) and $P$ ($F$) the secondary element (facet).
We assume that $\cPten$ is $1$-regular meaning that if $F'$ (a facet of $P'$) is a primary facet then it has non-empty intersection with secondary facets of $P'$. 
We stress that a refinement strategy that guarantees the latter conditions is discussed in~\cite[Section~4.1]{GantnerStevenson24}.

Based on the lowest-order version, i.e., $\ell=0$, $k=1$, we define 
\begin{align*}
  \cU_{\cPten} &= \set{(v,\btau,\mu)\in\bbS_{0,1}(\cPten)\times L^2(Q)}{v|_{(0,T)\times\partial\Omega} = 0, \, \mu|_P\in \bbP_{0,1}(P), \,\forall P\in \cPten}, \\
  \cKsymDiscTen &= \cKsymDiscTen^{\jop} = \set{(v,\btau,\mu)\in\cU_{\cPten}}{v(z) \geq \jop g(z), \,\forall z\in \cN_{\cPten}, \, \mu\geq 0}.
\end{align*}
Here, $\jop$ is a (quasi-)interpolation operator with $(\jop g,0)\in \bbS_{0,1}(\cPten)$.
We require that $\jop$ is feasibly computable. The simplest choice for $\jop$ is the nodal interpolation operator, i.e., $\jop g(z) = g(z)$ for all vertices $z\in \cN_{\cPten}$.
In the a priori analysis of Section~\ref{sec:apriori} below we consider a non-negativity preserving quasi-interpolator that allows for less regular solutions in the analysis.

\section{A priori and a posteriori analysis}\label{sec:aprioriposteriori}

\subsection{Quasi-interpolators for analysis on tensor mesh}\label{sec:quasi_inter_tensor}
For the proof of Theorem~\ref{thm:convergenceRatesGS} below we use some quasi-interpolation operators. 
First, we consider quasi-uniform tensor product meshes. 
Let us denote by $\cP_t$ a quasi-uniform partition in time with mesh-size $h_t$ and by $\cP_\bx$ a quasi-uniform simplicial partition in space with mesh-size $h_\bx$. Then, $\cPten = \cP_t\otimes \cP_\bx$.
Define $\iop_{\mathrm{Cl}} = \iop_{t,\mathrm{nod}}\circ \iop_{\bx,\mathrm{Cl}}$ where
$\iop_{t,\mathrm{nod}}\colon H^1( (0,T);H) \to \bbS_1(\cP_t;H)$ is the nodal interpolation in time where $\bbS_1(\cP_t;H)$ is the space of continuous piecewise polynomials of degree $\leq 1$ with values in the Hilbert space $H$ and $\iop_{\bx,\mathrm{Cl}}\colon H^k( (0,T);L^2(\Omega)) \to H^k( (0,T); \bbS_1(\cP_\bx)\cap H_0^1(\Omega))$ ($k\in\N$) is a weighted Cl\'ement quasi-interpolator from~\cite[Section~3]{FuehrerMixedFEMHm1} extended to the time interval. 
Here, $\bbS_1(\cP_\bx)$ is defined as $\bbS_1(\cPtri)$ with $Q$ replaced by $\Omega$. The operator is defined for all interior spatial vertices $z\in \cN_{\cP_\bx}\setminus\partial\Omega$ and a.e. $t\in(0,T)$ as follows.
With $\patch(\{z\})\subset \cP_\bx$ we denote the patch (vicinity) of node $z$, i.e., all its neighboring elements, and let $\Patch(\{z\})\subset \Omega$ denote the corresponding domain. Then, 
\begin{align*}
  (\iop_{\bx,\mathrm{Cl}}v)(t,z) := \int_{\Patch(\{z\})} v(t) \varphi_z \,\di\bx
\end{align*}
with weight function
\begin{align*}
  \varphi_z|_{K_\bx} = \begin{cases} 
    \frac{\alpha_{z,K_\bx}}{|K_\bx|} & K_\bx\in\patch(\{z\}), \\
    0 & \text{else},
  \end{cases}
\end{align*}
where for each $z\in \cN_\bx\setminus\partial\Omega$, coefficients $0\leq\alpha_{z,K_\bx} <1$ satisfy
\begin{align*}
  \sum_{K_\bx\in \patch(\{z\})} \alpha_{z,K_\bx} = 1, \qquad
  \sum_{K_\bx\in \patch(\{z\})} \alpha_{z,K_\bx}z_{K_\bx} = z.
\end{align*}
Here, $z_{K_\bx}\in K_\bx$ denotes the center of mass of $K_\bx$.

Recall that in the definition of $\cKsymDiscTen$ we need an operator $\jop$ to approximate the obstacle function.
For an appropriate definition of such an operator we extend $\iop_{\bx,\mathrm{Cl}}$ for boundary vertices as follows. 
Define $\widetilde\iop_{\bx,\mathrm{Cl}}$ by
\begin{align*}
  (\widetilde\iop_{\bx,\mathrm{Cl}}v)(t,z) = \begin{cases}
    (\iop_{\bx,\mathrm{Cl}}v)(t,z) & \text{if } z\in\cN_{\cP_\bx}\setminus\partial\Omega, \\
    v(t,z) & \text{if } z\in\cN_{\cP_\bx}\cap\partial\Omega,
  \end{cases}
  \quad\text{for almost all } t\in (0,T).
\end{align*}
Then, set $\widetilde\iop_{\mathrm{Cl}} = \iop_{t,\mathrm{nod}}\circ\widetilde\iop_{\bx,\mathrm{Cl}}$.
Note that $\widetilde\iop_{\mathrm{Cl}}g$ is well-defined. 

The following results summarize some properties of the operators defined above.
\begin{lemma}\label{lem:propCL}
  Operator $\iop_{\bx,\mathrm{Cl}}$ and $\iop_{\mathrm{Cl}}$ preserve non-negativity. 

  If $v\in L^2( (0,T);H^k(\Omega)\cap H_0^1(\Omega))$, then,
  \begin{align*}
    \|(1-\iop_{\bx,\mathrm{Cl}})v\|_{L^2(Q)} + h_\bx \|\gradx(1-\iop_{\bx,\mathrm{Cl}})v\|_{L^2(Q)}
    \lesssim h_\bx^k \|D_\bx^k v\|_{L^2(Q)} \qquad k=1,2.
  \end{align*}
  If $v\in H^1( (0,T); H_0^1(\Omega))$, then,
  \begin{align*}
    \|\partial_t(1-\iop_{\bx,\mathrm{Cl}})v\|_{\cV^*} \lesssim \|\partial_t(1-\iop_{\bx,\mathrm{Cl}})v\|_{L^2(Q)}  \lesssim h_\bx\|\partial_t\gradx v\|_{L^2(Q)}.
  \end{align*}
  If $v\in L^2( (0,T);H^k(\Omega)\cap H_0^1(\Omega))\cap H^j( (0,T);L^2(\Omega))$, then,
  \begin{align*}
    \|(1-\iop_{\mathrm{Cl}})v\|_{L^2(Q)} + h_\bx \|\gradx(1-\iop_{\mathrm{Cl}})v\|_{L^2(Q)} &\lesssim  h_\bx^k \|D_\bx^k v\|_{L^2(Q)} + h_t^j\|\partial_t^jv\|_{L^2(Q)} \qquad j,k=1,2.
  \end{align*}
\end{lemma}
\begin{proof}
  By definition all weight functions $\varphi_z$ are non-negative. Therefore, $\iop_{\bx,\mathrm{Cl}}v\geq 0$ for $v\geq 0$.
  Non-negativity of $\iop_{\mathrm{Cl}} = \iop_{t,\mathrm{nod}}\circ \iop_{\bx,\mathrm{Cl}}$ then follows since the nodal interpolation in time preserves non-negativity as well.
  The approximation estimates for $\iop_{\bx,\mathrm{Cl}}$ are consequences of~\cite[Theorem~11]{FuehrerMixedFEMHm1}.

  For the next estimate we use the rough inequality $\|\cdot\|_{\cV^*} \lesssim \|\cdot\|_{L^2(Q)}$ and $\partial_t(1-\iop_{\bx,\mathrm{Cl}})v = (1-\iop_{\bx,\mathrm{Cl}})\partial_t v$ to see that
  \begin{align*}
    \|\partial_t(1-\iop_{\bx,\mathrm{Cl}})v\|_{\cV^*} \leq \|(1-\iop_{\bx,\mathrm{Cl}})\partial_tv\|_{L^2(Q)} \lesssim h_\bx\|\gradx\partial_t v\|_{L^2(Q)}.
  \end{align*}

  Splitting the operator as in~\cite{FuehrerKarkulik21} we write $(1-\iop_{\mathrm{Cl}})v = (1-\iop_{\bx,\mathrm{Cl}})v + \iop_{\bx,\mathrm{Cl}}(1-\iop_{t,\mathrm{nod}})v$. Then, using that $\iop_{\bx,\mathrm{Cl}}$ is bounded from $L^2(Q)\to L^2(Q)$ (this again follows from~\cite[Theorem~11]{FuehrerMixedFEMHm1}) shows
  \begin{align*}
    \|(1-\iop_{\mathrm{Cl}})v\|_{L^2(Q)} &\lesssim \|(1-\iop_{\bx,\mathrm{Cl}})v\|_{L^2(Q)} + \|\iop_{\bx,\mathrm{Cl}}(1-\iop_{t,\mathrm{nod}})v\|_{L^2(Q)}
    \\
    &\lesssim h_\bx^k \|D_\bx^k v\|_{L^2(Q)} + \|(1-\iop_{t,\mathrm{nod}})v\|_{L^2(Q)} \lesssim h_\bx^k \|D_\bx^k v\|_{L^2(Q)} + h_t^j\|\partial_t^jv\|_{L^2(Q)}
  \end{align*}
  and with an additional inverse estimate the same ideas yield
  \begin{align*}
    \|\gradx (1-\iop_{\mathrm{Cl}})v\|_{L^2(Q)} &\lesssim \|\gradx(1-\iop_{\bx,\mathrm{Cl}})v\|_{L^2(Q)} + \|\gradx\iop_{\bx,\mathrm{Cl}}(1-\iop_{t,\mathrm{nod}})v\|_{L^2(Q)}
    \\
    &\lesssim h_\bx^{k-1} \|D_\bx^k v\|_{L^2(Q)} + h_\bx^{-1} \|(1-\iop_{t,\mathrm{nod}})v\|_{L^2(Q)} 
    \\
    &\lesssim h_\bx^{k-1} \|D_\bx^k v\|_{L^2(Q)} + h_t^jh_\bx^{-1}\|\partial_t^jv\|_{L^2(Q)}.
  \end{align*}
  This finishes the proof.
\end{proof}

\begin{remark}
  The first positivity preserving operator with minimal regularity requirements for the analysis of finite element methods for obstacle problems has been introduced in~\cite[Section~3]{ChenNochetto00}.
  There the authors define an operator $\iop_\mathrm{CN} v(z) = |B_z|^{-1}\int_{B_z} v\,\di\bx$ ($z\in\cN_{\cP_\bx}\setminus\partial\Omega$) where $B_z$ denotes the ball with maximal radius such that $B_z\subseteq \Omega(\{z\})$. This operator has the same approximation properties as $\iop_{\bx,\mathrm{Cl}}$ when extended in time and can be used in the definition of $\iop_\mathrm{Cl}$ instead of $\iop_{\bx,\mathrm{Cl}}$ in the analysis below.
  \qed
\end{remark}

We introduce further operators for the analysis.
Let $\Pi_{t,0}$ denote the $L^2(Q)$ orthogonal projection onto $\bbP_0(\cP_t;H)$ (the space of $\cP_t$-piecewise constant functions with values in some Hilbert space $H$) and let $\Pi_{\bx,k}$ denote the $L^2(Q)$ orthogonal projection in space onto $L^2( (0,T);\bbP_k(\cP_\bx))$.
Set $\Pi_{0,k} = \Pi_{\bx,k} \circ \Pi_{t,0} = \Pi_{t,0} \circ \Pi_{\bx,k}$.
Let $\iop_{\bx,\mathrm{nod}}\colon H^k( (0,T);H^2(\Omega))\to H^k( (0,T);\bbS_1(\cP_\bx))$ ($k\in\N$) the nodal interpolation operator in space and set $\iop_\mathrm{nod} = \iop_{t,\mathrm{nod}}\circ \iop_{\bx,\mathrm{nod}} = \iop_{\bx,\mathrm{nod}}\circ \iop_{t,\mathrm{nod}}$.

We also need the following operator. 
\begin{lemma}[{\cite[Theorem~17]{StevensonStorn23}}]\label{lem:iop2}
  There exists $\iop_{\divx}\colon L^2(Q)^d \to \mathbb{P}_{0}(\cP_t)\otimes \RT_1(\cP_\bx)\subset L^2(Q)^d$ linear and bounded satisfying the commutativity property
  \begin{align*}
    \divx \iop_{\divx}\btau = \Pi_{0,1}\divx\btau \qquad\forall \btau\in L^2( (0,T);\Hdivxset\Omega).
  \end{align*}
  Furthermore,
  \begin{align*}
    \|\btau-\iop_{\divx}\btau\|_{L^2(Q)}^2 &\eqsim \sum_{K_t\times K_\bx\in\cPten}
    \Big( 
      \min_{\btau_\bx\in L^2(K_t;\RT_1(\cP_\bx))} \|\btau-\btau_\bx\|_{L^2(K_t;L^2(\Patch(K_\bx)))}^2
      \\ 
      &\qquad + \min_{\btau_t \in \mathbb{P}_{0}(K_t;L^2(K_\bx))} \|\btau-\btau_t\|_{L^2(K_t;L^2(K_\bx))}^2
    \Big),
  \end{align*}
  where $\Patch(K_\bx):=\mathrm{int}\overline{\cup\{K_\bx^{\prime} \in \cP_\bx: \overline{K_\bx}|\cap 
  \overline{K_\bx^{\prime}} \neq \emptyset\}}$.
  \qed
\end{lemma}

Following the ideas of~\cite[Section~5.2]{StevensonStorn23} we define operator $\iop$ as follows
\begin{align*}
  \iop(v,\btau,\mu) &:= (\iop_1 (v,\btau,\mu),\iop_2(v,\btau,\mu),\iop_3(v,\btau,\mu)) \\
  &= (\iop_{\mathrm{Cl}}v,\iop_{\divx}\btau+\iop_{\divx}( \gradx\Delta_\bx^{-1}(\partial_t(v-\iop_{\bx,\mathrm{Cl}}v)-(\Pi_{0,1}\mu-\Pi_{0,0}\mu))),\Pi_{0,0}\mu).
\end{align*}
This operator is designed to have a commutativity property together with other properties as stated in the next result.
\begin{lemma}
  Operator $\iop(v,\btau,\mu)$ is well defined for all $(v,\btau,\mu)\in \cU$ with $\mu\in L^2(Q)$.
  It satisfies the commutativity property
  \begin{align*}
    \big(\divst(\iop_1,\iop_2)-\iop_3\big)(v,\btau,\mu) = \Pi_{0,1}(\divst(v,\btau)-\mu)
    \quad\forall (v,\btau,\mu)\in \cU \quad\text{with } \mu\in L^2(Q).
  \end{align*}
  Furthermore, 
  \begin{align*}\|\btau-\iop_2(v,\btau,\mu)\|_{L^2(Q)} &\lesssim \|\btau-\iop_{\divx}\btau\|_{L^2(Q)}
    + \|\partial_t(v-\iop_{\mathrm{\bx,Cl}}v)\|_{\cV^*} + h_\bx\|\mu\|_{L^2(Q)}.
  \end{align*}
\end{lemma}
\begin{proof}
  Let $(v,\btau,\mu)\in \cU$ with $\mu\in L^2(Q)$ be given. 
  Then,
  \begin{align*}
    (\divst(\iop_1,\iop_2)-\iop_3\big)(v,\btau,\mu)
    &= \partial_t\iop_{\mathrm{Cl}}v + \Pi_{0,1}\divx\btau + \Pi_{0,1}\partial_t(v-\iop_{\bx,\mathrm{Cl}}v) - \Pi_{0,1}\mu
    \\
    &= \Pi_{0,1}(\partial_t v + \divx\btau -\mu).
  \end{align*}
  For the last identity we have used that $\Pi_{0,1}\partial_t\iop_{\mathrm{\bx,Cl}}v = \Pi_{0,1}\Pi_{t,0}\partial_t \iop_{\bx,\mathrm{Cl}} v =\Pi_{0,1}\partial_t\iop_\mathrm{Cl}v=\partial_t \iop_{\mathrm{Cl}}$.
  This proves the commutativity. 
  
  Furthermore, triangle inequalities and the fact that $\Delta_\bx\colon \cV\to \cV^*$ is an isomorphism show that
  \begin{align*}
    \|\btau-\iop_{2}(v,\btau,\mu)\|_{L^2(Q)} &\lesssim \|\btau-\iop_{\divx}\btau\|_{L^2(Q)}
    + \|\partial_t(v-\iop_{\bx,\mathrm{Cl}}v)\|_{\cV^*} + \|(\Pi_{0,1}-\Pi_{0,0})\mu\|_{\cV^*}.
  \end{align*}
  To estimate the last term, we write $\Pi_{0,1}-\Pi_{0,0} = (\Pi_{\bx,1}-\Pi_{\bx,0})\Pi_{t,0}$ and use this identity to get
  \begin{align*}
    \dual{(\Pi_{0,1}-\Pi_{0,0})\mu}{v}_{\cV^*\times \cV} &= \ip{(\Pi_{\bx,1}-\Pi_{\bx,0})\Pi_{t,0}\mu}{v}_{L^2(Q)}
    = \ip{\Pi_{t,0}\mu}{(\Pi_{\bx,1}-\Pi_{\bx,0})v}_{L^2(Q)}
    \\
    &\leq \|\Pi_{t,0}\mu\|_{L^2(Q)} \|(\Pi_{\bx,1}-\Pi_{\bx,0})v\|_{L^2(Q)} \lesssim \|\mu\|_{L^2(Q)}h_\bx\|\gradx v\|_{L^2(Q)}.
  \end{align*}
  We conclude that $\|(\Pi_{0,1}-\Pi_{0,0})\mu\|_{\cV^*} \lesssim h_\bx \|\mu\|_{L^2(Q)}$.
\end{proof}

\subsection{A priori error analysis}\label{sec:apriori}

The following result, sometimes referred to as the Falk--C\'ea lemma, follows well-established arguments in the a priori analysis of variational inequalities, cf.~\cite{Falk74}.
For the case of least-squares methods for elliptic obstacle problems we refer to the proof of~\cite[Theorem~8]{LSQobstacle}.
The result below follows the same lines of argumentation and is therefore omitted.
\begin{proposition}\label{prop:error_estimate}
  Let $\bu=(u,\bsigma,\lambda)\in \cKsym$ denote the solution of~\eqref{eq:vi:sym} and let $\bu_\cP=(u_\cP,\bsigma_\cP,\lambda_\cP)\in \cKsymDisc$ denote the solution of~\eqref{eq:vi:sym:disc}. Then, 
  \begin{align*}
    |\bu-\bu_\cP|_\cU^2 &\lesssim \min_{\bv=(v,\btau,\mu)\in\cKsymDisc} \left(|\bu-\bv|_\cU^2 + \left\lvert\dual{\lambda}{v-u}_{\cV^*\times \cV} + \dual{\mu-\lambda}{u-g}_{\cV^*\times \cV}\right\rvert \right)
    \\
    &\qquad\quad + \min_{\bv=(v,\btau,\mu)\in \cKsym} \left\lvert \dual{\lambda}{v-u_{\cP}}_{\cV^*\times \cV} + \dual{\mu-\lambda_{\cP}}{u-g}_{\cV^*\times \cV} \right\rvert.
  \end{align*}
  \qed
\end{proposition}

\subsubsection{Convergence rates for Lagrange finite element discretization}
We consider the simplicial partition $\cP = \cPtri$ and assume that $\cP$ is quasi-uniform with mesh-size $h$ and is generated from a tensor product mesh, see, e.g.~\cite[Section~4.1.2]{FuehrerKarkulik21}.
The next result is stated without proof. It follows similar ideas as the proof of Theorem~\ref{thm:convergenceRatesGS} which we detail below. 
However, higher regularity assumptions are required for the statements in Theorem~\ref{thm:estimates_KF}. 
This is due to the fact that an interpolation operator satisfying properties similar to $\iop$ for tensor meshes can not be constructed for the Lagrange finite elements.
For parabolic PDEs without obstacle conditions we refer to~\cite{FuehrerKarkulik21,GantnerStevenson21}.

\begin{theorem}[Lagrange finite elements on simplicial meshes]
\label{thm:estimates_KF}
Suppose $\cP = \cPtri$, so that $\cU_\cP = \cU_{\cPtri}$ and $\cKsymDisc = \cKsymDiscTri$.
Let $\bu=(u,\bsigma,\lambda)\in \cKsym$ denote the solution of~\eqref{eq:vi:sym} and let $\bu_\cP=(u_\cP,\bsigma_\cP,\lambda_\cP)\in \cKsymDisc$ denote the solution of~\eqref{eq:vi:sym:disc}.
  Suppose that $g\in H^{2}(Q)$, $f\in H^{1}(Q)$, $\bA\in W^{1,\infty}(Q)^{d\times d}$, $\bb\in W^{1,\infty}(Q)^d$, $c\in W^{1,\infty}(Q)$, and $u\in H^1((0,T);H^2(\Omega))\cap L^\infty((0,T);H^3(\Omega))\cap H^2((0,T);L^2(\Omega))$, then
\begin{equation*}
|\bu-\bu_\cP|_\cU^2 
   \lesssim
 h^{2}\left(\|u\|_{H^{1}((0,T);H^{2}(\Omega))}^{2} + \|u\|_{L^{\infty}((0,T);H^{3}(\Omega))}^{2} + \|u\|_{H^{2}((0,T);L^{2}(\Omega))}^{2} + \|f\|_{H^{1}(Q)}^{2} + \|g\|_{H^{2}(Q)}^{2}\right). 
\end{equation*} 
\end{theorem}

\subsubsection{Convergence rates for Gantner--Stevenson finite element discretization}

In the next result we assume that the obstacle vanishes on the spatial boundary. 
For the case that $g <0$ in a subregion of ${(0,T)\times \partial\Omega}$ we employ the operator $\widetilde\iop_{\mathrm{Cl}}$ to approximate the obstacle.
\begin{theorem}[Gantner--Stevenson element on conforming prismatic meshes]\label{thm:convergenceRatesGS}
  Suppose $\cP = \cPten$, so that $\cU_\cP = \cU_{\cPten}$ and $\cKsymDisc = \cKsymDiscTen^\jop$ with $\jop = \iop_{\mathrm{Cl}}$ and $g|_{(0,T)\times\partial\Omega} = 0$.
Let $\bu=(u,\bsigma,\lambda)\in \cKsym$ denote the solution of~\eqref{eq:vi:sym} and let $\bu_\cP=(u_\cP,\bsigma_\cP,\lambda_\cP)\in \cKsymDisc$ denote the solution of~\eqref{eq:vi:sym:disc}.
Suppose that $g \in L^2( (0,T);H^2(\Omega))\cap H^j( (0,T);L^2(\Omega))$, $\bA\in W^{1,\infty}(Q)^{d\times d}$, $\bb\in W^{1,\infty}(Q)^{d}$, $c\in W^{1,\infty}(Q)$,
and $u\in H^1( (0,T);H^1(\Omega))\cap L^2( (0,T);H^2(\Omega))\cap H^j( (0,T);L^2(\Omega))$ for $j=1$ or $j=2$.
Then, 
\begin{align}\label{eq:GSrates}
\begin{split}
  |\bu-\bu_\cP|_\cU &\lesssim  h_\bx(\|D_\bx^2u\|_{L^2(Q)}  + \|\gradx u\|_{L^2(Q)}+ \|D_\bx^2 g\|_{L^2(Q)} + \|\partial_t\gradx u\|_{L^2(Q)} + \|\lambda\|_{L^2(Q)}) 
  \\
  &\quad + h_t^{j/2}(\|\partial_t^j u\|_{L^2(Q)} + \|\partial_t^j g\|_{L^2(Q)} + \|\lambda\|_{L^2(Q)})
  \\
  &\quad + \frac{h_t^j}{h_\bx} \|\partial_t^j u\|_{L^2(Q)} + h_t(\|\partial_t\gradx u\|_{L^2(Q)} + \|\gradx u\|_{L^2(Q)})
  + \|(1-\Pi_{0,1})f\|_{L^2(Q)}.
\end{split}
\end{align}
\end{theorem}

Before we give a proof of Theorem~\ref{thm:convergenceRatesGS} some remarks are in order. 

\begin{remark}
  If $f$ is elementwise regular, i.e., $f|_{K_t\times K_\bx} \in H^1(K_t;L^2(K_\bx))\cap L^2( K_t;H^k(K_\bx))$ for all $K_t\times K_\bx\in\cPten$, where either $k=1$ or $k=2$, then
  \begin{align*}
    \|(1-\Pi_{0,1})f\|_{L^2(Q)} \lesssim h_\bx^k \Big(\sum_{P\in\cPten} \|D_\bx^k f\|_{L^2(P)}^2\Big)^{1/2}
    + h_t \Big(\sum_{P\in\cPten} \|\partial_t f\|_{L^2(P)}^2\Big)^{1/2}.
  \end{align*}
\end{remark}
\begin{remark}
  If $j=1$ in Theorem~\ref{thm:convergenceRatesGS} then estimate~\eqref{eq:GSrates} suggests parabolic scaling $h_t \eqsim h_\bx^2$.

  Some estimates in the proof below are not sharp and that is why we assume additional regularity, e.g., $\partial_t\gradx u \in L^2(Q)^d$.
  Using the operator from~\cite[Section~5.1]{StevensonStorn23} (denoted $\iop_X^\otimes$ there) instead of $\iop_\mathrm{Cl}$ would get rid of these assumptions. When estimating $\|\cdot\|_{\cV^*}$ with $\|\cdot\|_{L^2(Q)}$ we loose a possible power of $h_\bx$, i.e.,
  \begin{align*}
    \|\partial_t(1-\iop_{\bx,\mathrm{Cl}})u\|_{\cV^*} \leq \|(1-\iop_{\bx,\mathrm{Cl}})\partial_tu\|_{L^2(Q)} \lesssim h_\bx\|\partial_t\gradx u\|_{L^2(Q)}
  \end{align*}
  while for the operator denoted with $\iop_x$ in~\cite{StevensonStorn23} one can show that
  \begin{align*}
    \|\partial_t(1-\iop_{x})u\|_{\cV^*} \lesssim h_\bx\|(1-\iop_{x})\partial_tu\|_{L^2(Q)} \lesssim h_\bx\|\partial_t u\|_{L^2(Q)}.
  \end{align*}
  However, $\iop_x$ and therefore $\iop_X^\otimes = \iop_{t,\mathrm{nod}}\circ\iop_x$ do not preserve non-negativity which is a crucial property in our analysis of the variational inequality.
\end{remark}
\begin{remark}
  In the analysis operator $\iop_{\mathrm{Cl}}$ could be replaced by the nodal interpolation $\iop_{\mathrm{nod}}$ but requires higher regularity assumptions (for $d>1$). To see this,  we note that we have to estimate
  \begin{align*}
    \|\partial_t(u-\iop_{\bx,\mathrm{nod}}u)\|_{\cV^*} \leq \|(1-\iop_{\bx,\mathrm{nod}})\partial_tu\|_{L^2(Q)} \lesssim h_\bx^2\|\partial_t D_\bx^2u\|_{L^2(Q)}.
  \end{align*}
  Thus, we would additionally require that $u \in H^1( (0,T); H^2(\Omega))$ instead of $u\in L^2( (0,T);H^2(\Omega))$ in Theorem~\ref{thm:convergenceRatesGS}.
\end{remark}

\begin{proof}[Proof of Theorem~\ref{thm:convergenceRatesGS}]
  In view of the estimate from Proposition~\ref{prop:error_estimate} 
  we need to establish a bound for
  \begin{align}\label{eq:proof:convrates1}
    |\bu-\bv|_\cU^2 + \left\lvert\dual{\lambda}{v-u}_{\cV^*\times \cV} + \dual{\mu-\lambda}{u-g}_{\cV^*\times \cV}\right\rvert
  \end{align}
  for some $\bv=(v,\btau,\mu)\in\cKsymDisc$. To do so, we choose $\bv= \iop(u,\bsigma,\lambda)$. 
  Note that $\iop_{\mathrm{Cl}}u(z) \geq \iop_{\mathrm{Cl}}g(z)$ for all nodes of the partition $\cPten$ and $\Pi_{0,0}\lambda\geq 0$ since $\lambda\geq 0$. Therefore, $\bv\in \cKsymDisc$.
  The results on the quasi-interpolation operators then yield
  \begin{align*}
    |\bu-\bv|_\cU &\lesssim \|\gradx(1-\iop_{\mathrm{Cl}})u\|_{L^2(Q)} + \|(1-\iop_{2})\bsigma\|_{L^2(Q)}  + \|\partial_t(1-\iop_{\mathrm{Cl}})u\|_{\cV^*}\\
    & \quad \quad + \|(1-\Pi_{0,1})(\divst(u,\bsigma)-\lambda)\|_{L^2(Q)}
    \\
    &\lesssim h_\bx\|D_\bx^2 u\|_{L^2(Q)} + \frac{h_t^j}{h_\bx}\|\partial_t^ju\|_{L^2(Q)} 
    + h_\bx \|\gradx \bsigma\|_{L^2(Q)} + h_t\|\partial_t\bsigma\|_{L^2(Q)} + h_\bx\|\partial_t \gradx u\|_{L^2(Q)} 
   \\
   &\qquad + h_\bx\|\lambda\|_{L^2(Q)}
   + \|(1-\Pi_{0,1})f\|_{L^2(Q)}  + \|(1-\Pi_{0,1})(\bb\cdot\gradx u + c u)\|_{L^2(Q)}.
  \end{align*}
  In the last estimate we have also used that $\divst(u,\bsigma)-\lambda = f-\bb\cdot\gradx u -cu$.
  The last term on the right-hand side is further estimated with
  \begin{align*}
   & \|(1-\Pi_{0,1})(\bb\cdot\gradx u + c u)\|_{L^2(Q)} \\
   & \leq \|(1-\Pi_{t,0})(\bb\cdot\gradx u + c u)\|_{L^2(Q)} + \|\Pi_{t,0}(1-\Pi_{\bx,1})(\bb\cdot\gradx u + c u)\|_{L^2(Q)} \\
    & \lesssim h_t(\|\partial_t\gradx u\|_{L^2(Q)} + \|\partial_t u\|_{L^2(Q)} + \|\gradx u\|_{L^2(Q)} + \|u\|_{L^2(Q)}) \\
    & \quad + h_\bx (\|D_\bx^2 u\|_{L^2(Q)} + \|\gradx u\|_{L^2(Q)} + \|u\|_{L^2(Q)}).
  \end{align*}
  This proves that the norm in~\eqref{eq:proof:convrates1}
  can be bounded as follows,
  \begin{align*}
    |\bu-\bv|_\cU &\lesssim h_\bx\|D_\bx^2 u\|_{L^2(Q)} + \frac{h_t^j}{h_\bx}\|\partial_t^ju\|_{L^2(Q)} 
    + (h_\bx+h_t) \|\partial_t\gradx u\|_{L^2(Q)} + h_\bx\|\lambda\|_{L^2(Q)}
   \\
   &\qquad + (h_t + h_\bx)\|\gradx u\|_{L^{2}(Q)} + \|(1-\Pi_{0,1})f\|_{L^2(Q)}.
  \end{align*}
  Next, we bound the duality terms $|\dual{\lambda}{\iop_{\mathrm{Cl}}u-u}_{\cV^*\times \cV}|$ and $|\dual{\Pi_{0,0}\lambda-\lambda}{u-g}_{\cV^*\times \cV}|$.
  First, 
  \begin{align*}
    |\dual{\lambda}{\iop_{\mathrm{Cl}}u-u}_{\cV^*\times \cV}|
    &\leq \|\lambda\|_{L^2(Q)}\|(1-\iop_{\mathrm{Cl}})u\|_{L^2(Q)} \lesssim (h_\bx^2\|D_\bx^2 u\|_{L^2(Q)} + h_t^j\|\partial_t^j u\|_{L^2(Q)})\|\lambda\|_{L^2(Q)}.
  \end{align*}
  To estimate $|\dual{\Pi_{0,0}\lambda-\lambda}{u-g}_{\cV^*\times \cV}|$ we write $(1-\Pi_{0,0})\lambda = (1-\Pi_{\bx,0})\lambda + (1-\Pi_{t,0})\Pi_{\bx,0}\lambda$.
  With this splitting we further estimate
  \begin{align*}
    |\dual{\Pi_{0,0}\lambda-\lambda}{u-g}_{\cV^*\times \cV}| &\leq \|\lambda\|_{L^2(Q)}\|\Pi_{\bx,0}(1-\Pi_{t,0})(u-g)\|_{L^2(Q)}
    + |\ip{(1-\Pi_{\bx,0})\lambda}{u-g}_{L^2(Q)}|
    \\ 
    &\lesssim h_t\|\partial_t(u-g)\|_{L^2(Q)}\|\lambda\|_{L^2(Q)} + |\ip{(1-\Pi_{\bx,0})\lambda}{u-g}_{L^2(Q)}|.
  \end{align*}
  Following the same lines of argumentation as in~\cite[Proof of Theorem~13]{LSQobstacle} one shows that
  \begin{align*}
    |\ip{(1-\Pi_{\bx,0})\lambda(t)}{(u-g)(t)}_{L^2(\Omega)}| \lesssim h_\bx^2\|D_\bx^2(u-g)(t)\|_{L^2(\Omega)}\|\lambda(t)\|_{L^2(\Omega)}
    \quad\text{for a.e. }t\in(0,T).
  \end{align*}
  We conclude that
  \begin{align*}
    |\dual{\Pi_{0,0}\lambda-\lambda}{u-g}_{\cV^*\times \cV}| &\lesssim h_\bx^2(\|D_\bx^2(u-g)\|_{L^2(Q)}^2 + \|\lambda\|_{L^2(Q)}^2)
      + h_t (\|\partial_t(u-g)\|_{L^2(Q)}^2 + \|\lambda\|_{L^2(Q)}^2).
  \end{align*}
  In view of Proposition~\ref{prop:error_estimate} it only remains to estimate
  \begin{align*}
    \min_{\bv=(v,\btau,\mu)\in \cKsym} \left\lvert \dual{\lambda}{v-u_{\cP}}_{\cV^*\times \cV} + \dual{\mu-\lambda_{\cP}}{u-g}_{\cV^*\times \cV} \right\rvert.
  \end{align*}
  Set $\bv = (\max\{g,u_\cP\},0,\lambda_\cP)$ and note that $\bv\in\cKsym$. Therefore, 
  \begin{align*}
    \left\lvert \dual{\lambda}{v-u_{\cP}}_{\cV^*\times \cV} + \dual{\mu-\lambda_{\cP}}{u-g}_{\cV^*\times \cV} \right\rvert
    &= \left\lvert \dual{\lambda}{\max\{g,u_\cP\}-u_{\cP}}_{\cV^*\times \cV}\right\rvert
    \\
    &\lesssim \|\lambda\|_{L^2(Q)} \|\max\{g-u_\cP,0\}\|_{L^2(Q)}.
  \end{align*}
Since $u_\cP\geq \iop_{\mathrm{Cl}}g$ we have $|\max\{g-u_\cP,0\}|\leq |g-\iop_{\mathrm{Cl}}g|$ on $Q$
and, consequently, 
  \begin{align*}
    \|\max\{g-u_\cP,0\}\|_{L^2(Q)} &\lesssim \|(1-\iop_{\mathrm{Cl}})g\|_{L^2(Q)}
    \lesssim h_\bx^2\|D_\bx^2 g\|_{L^2(Q)} + h_t\|\partial_t g\|_{L^2(Q)}.
  \end{align*}
  Combining all previous estimates together with Proposition~\ref{prop:error_estimate} finishes the proof.
\end{proof}

\subsection{A posteriori error estimator}\label{sec:aposteriori}
In this section, we define an a posteriori error estimator and prove its reliability for a discrete solution $(u_\cP,\bsigma_\cP,\lambda_\cP)$. 
Define the error estimators 
\begin{align*}
  \rho_p^2 &:= \|\bA^{1/2}\gradx(g-u_\cP)^+\|_{L^2(Q)}^2 + \left(\frac{\alpha}{\Lambda} + \frac{\|\bb\|_{L^\infty(Q)}^2 + \|c\|_{L^\infty(Q)}^2\Lambda}{\alpha} \right)\|(g-u_\cP)^+\|_{L^2(Q)}^2 \\
  &\qquad\qquad + \alpha^{-1}\|\partial_t(g-u_\cP)^+\|_{\cV^*}^2 + \|(g-u_\cP)^+(0)\|_{L^2(\Omega)}^2 + \|(g-u_\cP)^+(T)\|_{L^2(\Omega)}^2, \\
  \rho_c^2 &:= \ip{\lambda_\cP}{(u_\cP-g)^+}_{L^2(Q)},
\end{align*}
where $\rho_{p}$ measures the penetration of the obstacle contact condition and $\rho_{c}$ measures a violation of the obstacle complementarity condition for the discrete solution components $u_{\cP}$, $\lambda_{\cP}$.
Finally, we introduce the residual term
\begin{align*}
  \rho_{r}^{2}  := &\, \frac{\Lambda}{\alpha} \|f- \divst(u_{\mathcal{P}},\bsigma_{\cP}) - \bb\cdot\nabla_{\bx}u_{\cP} - cu_{\mathcal{P}} + \lambda_{\cP} \|_{L^{2}(Q)}^{2} \\
 &+ \|\bA^{-1/2}\bsigma_{\mathcal{P}} + \bA^{1/2}\nabla_{\bx}u_{\mathcal{P}}\|_{L^{2}(Q)}^{2} + \|u_{0} - u_{\cP}(0)\|_{L^{2}(\Omega)}^{2}.
\end{align*}

We are ready to present and prove the main result of this section.

\begin{theorem}[reliability]\label{thm:reliability}
  Let $\bu=(u,\bsigma,\lambda)\in \cKsym$ denote the solution of~\eqref{eq:vi:sym} and let $\bu_\cP=(u_\cP,\bsigma_\cP,\lambda_\cP)\in \cKsymDisc$ denote the solution of~\eqref{eq:vi:sym:disc}. Then, 
  \begin{align*}
   |\bu-\bu_\cP|_\cU^2 \lesssim \rho_{r}^{2} + \rho_{p}^{2} + \rho_{c}^{2}.
  \end{align*}
  The hidden constant is independent of continuous and discrete solutions.
\end{theorem}
\begin{proof}
  In the proof we use the (squared) norm %
    \begin{align*}
      \enorm{\bu}_{\cU}^2 &= \|\bA^{1/2}\gradx u\|_{L^2(Q)}^2 + \|\bA^{-1/2}\bsigma\|^2 + \|u(0)\|_{L^2(\Omega)}^2 + \|u(T)\|_{L^2(\Omega)}^2
  \\
  &\qquad + \frac{\Lambda}\alpha\|\divst(u,\bsigma)+\bb\cdot\gradx u + cu -\lambda\|_{L^2(Q)}^2
  \end{align*}
which is equivalent to $|\bu|_{\cU}$ as can be seen from the arguments given in proof of Lemma~\ref{lem:propblfsym}. 
  From~\eqref{eq:proof:coercivity1}, the identities $\bsigma+\bA\gradx u = 0$, $u(0) = u_0$ it follows that
    \begin{align}\label{eq:rel_es_dual}\begin{split}
        \enorm{\bu-\bu_\cP}_{\cU}^2 \lesssim a(\bu-\bu_\cP,\bu-\bu_\cP) = \rho_{r}^{2} + \dual{\lambda - \lambda_{\cP}}{u - u_{\cP}}_{\cV^*\times \cV},
\end{split}
\end{align}
where the generic constant hidden in ``$\lesssim$'' is, particularly, independent of $\Lambda$, $\bA$, $\bb$, $c$, $\alpha$.
We write 
$$
\dual{\lambda - \lambda_{\cP}}{u - u_{\cP}}_{\cV^*\times \cV} 
=
\dual{\lambda}{u - u_{\cP}}_{\cV^*\times \cV} + \ip{\lambda_{\cP}}{u_{\cP} - u}_{L^2(Q)}
=: \mathsf{I} + \mathsf{II},
$$ 
and bound each term separately. 
To estimate $\mathsf{II}$, we use that $\lambda_{\cP}, u - g \geq 0$ in $Q$ to obtain
\begin{align*}
\mathsf{II} 
=
\ip{\lambda_{\cP}}{u_{\cP} - g}_{L^2(Q)} + \ip{\lambda_{\cP}}{g - u}_{L^2(Q)}
\leq 
\ip{\lambda_{\cP}}{u_{\cP} - g}_{L^2(Q)}.
\end{align*}
On the other hand, using $\lambda\geq 0$, $\dual{\lambda}{u-g}_{\cV^*\times \cV}=0$, 
$g\leq \max\{g,u_\cP\}\in \cV$ we have that
\begin{align*}
\mathsf{I}
= & \,
\dual{\lambda}{u - g}_{\cV^*\times \cV} + \dual{\lambda}{g - \max\{g,u_{\cP}\}}_{\cV^*\times \cV}  + \dual{\lambda}{\max\{g,u_{\cP}\} - u_{\cP}}_{\cV^*\times \cV} \\
\leq & \,
\dual{\lambda}{\max\{g,u_\cP\}  - u_\cP}_{\cV^*\times \cV} = \dual{\lambda}{(g - u_{\cP})^+}_{\cV^*\times \cV}  \\
=  & \,
\dual{\lambda - \lambda_{\cP}}{(g - u_{\cP})^+}_{\cV^*\times \cV}  + \dual{\lambda_{\cP}}{(g - u_{\cP})^+}_{\cV^*\times \cV}.
\end{align*}
Consequently, from the estimates for $\mathsf{I}$ and $\mathsf{II}$ it follows that 
\begin{align*}
\dual{\lambda - \lambda_{\cP}}{u - u_{\cP}}_{\cV^*\times \cV} 
\leq & \,
\dual{\lambda - \lambda_{\cP}}{(g - u_{\cP})^+}_{\cV^*\times \cV} + \ip{\lambda_{\cP}}{(g - u_{\cP})^+}_{L^2(Q)} + \ip{\lambda_{\cP}}{u_{\cP} - g}_{L^2(Q)} \\
= & \, 
\dual{\lambda - \lambda_{\cP}}{(g - u_{\cP})^+}_{\cV^*\times \cV} + \ip{\lambda_\cP}{(u_\cP-g)^+}_{L^2(Q)} \\
&=\dual{\lambda - \lambda_{\cP}}{(g - u_{\cP})^+}_{\cV^*\times \cV} + \rho_{c}^{2}. 
\end{align*}
It remains to estimate $\dual{\lambda - \lambda_{\cP}}{(g - u_{\cP})^+}_{\cV^*\times \cV}$. 
Using Cauchy--Schwarz and Young inequalities gives
\begin{align}\label{eq:proof:apost:2}\begin{split}
& \dual{\lambda - \lambda_{\cP}}{(g - u_{\cP})^+}_{\cV^*\times \cV} \\
& =
\dual{-\div_{t,\bx}(u-u_{\cP},\bsigma - \bsigma_{\cP})-\bb\cdot\gradx(u-u_\cP)-c(u-u_\cP) + (\lambda - \lambda_{\cP})}{(g - u_{\cP})^+}_{\cV^*\times \cV} 
\\
&\qquad +\dual{\div_{t,\bx}(u-u_{\cP},\bsigma - \bsigma_{\cP})+\bb\cdot\gradx(u-u_\cP)-c(u-u_\cP)}{(g - u_{\cP})^+}_{\cV^*\times \cV} \\
& \leq \frac12\rho_r^2 + \frac{\alpha}{2\Lambda}\|(g - u_{\cP})^+\|_{L^{2}(Q)}^{2} + \dual{\partial_{t}(u-u_{\cP})}{(g - u_{\cP})^+}_{\cV^*\times \cV} \\
& \quad  +\dual{\div_{\bx}(\bsigma - \bsigma_{\cP})}{(g - u_{\cP})^+}_{\cV^*\times \cV} + \ip{\bb\cdot\gradx(u-u_\cP)+c(u-u_\cP)}{(g-u_\cP)^+}_{L^2(Q)}.
\end{split}
\end{align}
In what follows we estimate the last three terms on the right-hand side of~\eqref{eq:proof:apost:2} separately where for each one we use Young's inequality with parameter $\delta>0$.
For the first term integration by parts in time and $\alpha^{1/2}\|\cdot\|_{\cV} \leq \|\bA^{1/2}\gradx(\cdot)\|_{L^2(Q)}$ yields 
\begin{align*}
  \dual{\partial_t(u-u_\cP)}{(g-u_\cP)^+}_{\cV^*\times \cV} &= -\dual{(u-u_\cP)}{\partial_t(g-u_\cP)^+}_{\cV\times \cV^*}
  + \ip{u-u_\cP}{(g-u_{\cP})^+}_{L^2(\Omega)}|_0^T
  \\
  &\leq \alpha^{-1/2}\|\partial_t(g-u_\cP)^+\|_{\cV^*}\|\bA^{1/2}\gradx(u-u_\cP)\|_{L^2(Q)} \\
  &\qquad+ \|(u-u_\cP)(0)\|_{L^2(\Omega)}\|(g-u_\cP)^+(0)\|_{L^2(\Omega)} \\
  &\qquad+ \|(u-u_\cP)(T)\|_{L^2(\Omega)}\|(g-u_\cP)^+(T)\|_{L^2(\Omega)} \\
  &\leq \frac{\delta^{-1}}{2\alpha} \|\partial_t(g-u_\cP)^+\|_{\cV^*}^2 + \frac{\delta}2 \enorm{\bu-\bu_\cP}^2
  \\
  &\qquad + \frac{\delta^{-1}}2(\|(g-u_\cP)^+(0)\|_{L^2(\Omega)}^2+\|(g-u_\cP)^+(T)\|_{L^2(\Omega)}^2).
\end{align*}
Then, the definition of the divergence operator proves
\begin{align*}
  \dual{\divx(\bsigma-\bsigma_\cP)}{(g-u_\cP)^+}_{\cV^*\times \cV} &= \ip{\bsigma-\bsigma_{\cP}}{\gradx(g-u_\cP)^+}_{L^2(Q)}
  \\
  &\leq \|\bA^{-1/2}(\bsigma-\bsigma_\cP)\|_{L^2(Q)}\|\bA^{1/2}\gradx(g-u_\cP)^+\|_{L^2(Q)}
  \\
  &\leq \frac{\delta}2\enorm{\bu-\bu_\cP}_{\cU}^2 + \frac{\delta^{-1}}2\|\bA^{1/2}\gradx(g-u_\cP)^+\|_{L^2(Q)}^2.
\end{align*}
The final term on the right-hand side of~\eqref{eq:proof:apost:2} is estimated by
\begin{align*}
  &\ip{\bb\cdot\gradx(u-u_\cP)+c(u-u_\cP)}{(g-u_\cP)^+}_{L^2(Q)} 
  \\
  &\qquad\leq \alpha^{-1/2}\|\bb\|_{L^\infty(Q)}\|(g-u_\cP)^+\|_{L^2(Q)} \|\bA^{1/2}\gradx(u-u_\cP)\|_{L^2(Q)} \\
  &\qquad\qquad + \Lambda^{1/2}\alpha^{-1/2}\|(g-u_\cP)^+\|_{L^2(Q)}\|c\|_{L^\infty(Q)} \|\bA^{1/2}\gradx(u-u_\cP)\|_{L^2(Q)} \\
  &\qquad \leq \frac{\delta^{-1}}2\frac{\|\bb\|_{L^\infty(Q)}^2 + \Lambda\|c\|_{L^\infty(Q)}^2}{\alpha}\|(g-u_\cP)^+\|_{L^2(Q)}^2 + \delta\enorm{\bu-\bu_\cP}_\cU^2.
\end{align*}
Combining all the last estimates together with~\eqref{eq:rel_es_dual} and~\eqref{eq:proof:apost:2} shows that
\begin{align*}
  \enorm{\bu-\bu_\cP}_{\cU}^2 \lesssim \rho_r^2 + \delta^{-1} \rho_p^2 + \rho_c^2 + \delta\enorm{\bu-\bu_\cP}_\cU^2.
\end{align*}
Note that the hidden constant is independent of any quantity of interest. In particular, the proof is finished by subtracting the last term on the right-hand side for sufficiently small $\delta$.
\end{proof}

\begin{remark}\label{rem:estimator}
  We note that the estimator contribution $\|\partial_t(g-u_\cP)^+\|_{\cV^*}$ is, in general, not computable. 
  In some of the numerical examples from Section~\ref{sec:numerics} we replace $\rho$ by 
  \begin{align*}
    \widetilde\rho^2 &= \rho_r^2 + \rho_c^2 + \widetilde\rho_p^2, \\
    \widetilde\rho_p^2 &= \|\bA^{1/2}\gradx(g-u_\cP)^+\|_{L^2(Q)}^2 + \left(\frac{\alpha}{\Lambda} + \frac{\|\bb\|_{L^\infty(Q)}^2 + \|c\|_{L^\infty(Q)}^2\Lambda}{\alpha} \right)\|(g-u_\cP)^+\|_{L^2(Q)}^2 \\
    &\qquad\qquad + \frac{\Lambda}\alpha\|\partial_t(g-u_\cP)^+\|_{L^2(Q)}^2 + \|(g-u_\cP)^+(0)\|_{L^2(\Omega)}^2 + \|(g-u_\cP)^+(T)\|_{L^2(\Omega)}^2
  \end{align*}
  Note that $\rho\leq \widetilde\rho$. Therefore, $\widetilde\rho$ is a reliable, computable and localizable error estimator. 

  Alternatively, one could seek an approximation $w_\cP$ to the auxiliary problem
  \begin{align*}
      w\in \cV\colon\quad   -\Delta_\bx w = \partial_t(g-u_\cP)^+ \quad\text{in } Q,
  \end{align*}
and then use $\|w_\cP\|_{\cV}$ instead of $\|\partial_t(g-u_\cP)^+\|_{\cV^*}$. Such an approach requires taking into account an additional approximation error.

  Another possibility is to consider weighted norms. We have used such an idea for the numerical example presented in Section~\ref{sec:blackscholes}.
  \qed
\end{remark}

\begin{remark}
  In the proof of Theorem~\ref{thm:reliability} we actually have not used that $(u_\cP,\bsigma_\cP,\lambda_\cP)\in\cKsymDisc$ is the solution to~\eqref{eq:vi:sym:disc}. In fact, we only employed that $(u,\bsigma,\lambda)\in\cKsym$ solves~\eqref{eq:fo} and that $\lambda_\cP\geq 0$. 
  Therefore, $(u_\cP,\bsigma_\cP,\lambda_\cP)$ in the definition of the estimator and in the estimate in Theorem~\ref{thm:reliability} can be replaced by any $(v,\btau,\mu)\in \cU$ with $\mu\in L^2(Q)$ and $\mu\geq 0$.
  \qed
\end{remark}

\section{Numerical experiments}\label{sec:numerics}
In this section we present some numerical examples that exhibit the performance of the proposed LS-FEM. 
The experiments were performed with codes implemented in MATLAB.

We use the bulk criterion 
\begin{align*}
  \theta \tilde{\rho}^2 \leq \sum_{P\in\mathcal{M}} \tilde{\rho}(P)^2
\end{align*}
to determine a (minimal) set of elements $\mathcal{M}\subset\cP$ that are marked for refinement. 
For realizing the mesh-refinements of simplicial meshes, we employ the newest vertex bisection algorithm.
For (local) refinements of tensor meshes we use the method described in~\cite{GantnerStevenson24}, see also Section~\ref{sec:gsfem} above.
The adaptive loop consists of repeating the four major standard steps, \emph{Solve}, \emph{Estimate}, \emph{Mark}, \emph{Refine}.
The parameter $\theta = \tfrac12$ is chosen in the bulk criterion.

Discrete variational inequality~\eqref{eq:vi:sym:disc} can be written for both discretizations in the form
  \begin{align}\label{eq:VI:MatrixForm}
  \text{Find }\boldsymbol{x}\in\boldsymbol{K}\colon \quad
    \ip{\boldsymbol{S}\boldsymbol{x}}{\boldsymbol{y}-\boldsymbol{x}}_2 \geq \ip{\boldsymbol{F}}{\boldsymbol{y}-\boldsymbol{x}}_2
    \quad\forall \boldsymbol{y}\in\boldsymbol{K},
\end{align}
where $\boldsymbol{S}\in \R^{N\times N}$ is the Galerkin matrix of $\blfsym(\cdot,\cdot)$ with respect to the canonic basis of $\cU_\cP$, $\boldsymbol{F}\in\R^N$ is the discretization of $\rhssym$ with respect to the same basis, $\boldsymbol{K}\subset \R^N$ is a convex set corresponding to $\cKsymDisc$ (pointwise constraints).
Variational inequality~\eqref{eq:VI:MatrixForm} is a prototypical problem in convex optimization and several solvers exist for this type. 
Here, we employ the semi-smooth Newton method from~\cite{MR1972219}.

\subsection{One-phase Stefan problem}
Let $\Omega = (0,1)$ and $T=1$, hence, $Q = (0,1)^2$. 
We consider the solution $u(t,x)$ of the system
\begin{alignat*}{2}
  \partial_t u - \Delta_x u &\geq f &\quad&\text{in } Q, \\
  u\geq g, \quad (\partial_t u - \Delta_x u-f)(u-g) &= 0 &\quad&\text{in } Q, \\
  u(0,x) &= u_0, &\quad&\text{for } x\in \Omega,\\
  u(t,0) = u(t,1) &= 0 &\quad&\text{for }t\in (0,1)
\end{alignat*}
with data 
\begin{align*}
  h(t) &= e^t-t-1, \quad f(t,x) =-1-\partial_th(t)(1-x), \\
  u_0(x) &= -h(0)(1-x) = 0, \quad g(t,x) = -h(t)(1-x).
\end{align*}
The exact solution is given by
\begin{align*}
  u(t,x) = -h(t)(1-x) + \begin{cases}
    e^{t-x}+x-t-1 & \text{if }t>x, \\
    0 & \text{else}.
  \end{cases}
\end{align*}
We note that the coincidence set, i.e., $\{u=g\}$, is given by $\set{(t,x)\in Q}{t\leq x}$.

Furthermore, the function $\widetilde u = u + h(t)(1-x)$ is related via Duvaut's transformation $\widetilde u(t,x) = \int_0^t \Theta(s,x)\,\di s$~\cite{DuvautTrafo}
to the following one-phase Stefan problem:
\begin{align*}
  \partial_t \Theta -\Delta_x \Theta &= 0 \quad\text{for } 0<x<s(t), \, t>0, \qquad \Theta= 0 \quad\text{for } s(t)\leq x < 1, \, t>0, \\
  \Theta(0,x) &= 0 \quad\text{for } x\in (0,1), \\
  s(0) &= 0, \qquad \frac{\di s}{\di t} = -\partial_x \Theta(t,s(t)), \\
  \Theta(t,0) &= e^t-1, \qquad  \Theta(t,1) = 0, \, t>0.
\end{align*}
Note that in our situation one easily verifies that the phase separation interface is given by $s(t) = t$, see also~\cite[Section~2]{MitchellVynnycky09} for the presented and other explicit solutions of the Stefan problem.

A straightforward calculation proves that $u\in H^2(Q)$ and $f\in C^\infty(\overline Q)$. In view of Theorem~\ref{thm:convergenceRatesGS} we expect 
\begin{align*}
  |\bu-\bu_\cP|_\cU = \OO(h_\bx+h_t)
\end{align*}
when using tensor product meshes $\cP=\cPten$.
The left plot of Figure~\ref{fig:stefan}, which shows the total error compared to the overall estimator $\widetilde\rho$, confirms this behavior. 

The estimate from Theorem~\ref{thm:estimates_KF} for simplicial meshes $\cP=\cPtri$ assumes that the solution is in the space $L^\infty( (0,T); H^3(\Omega))$ which is not the case for the situation at hand. Nevertheless, we observe optimal rates, i.e., $|\bu-\bu_\cP|_\cP = \OO(h)$, see right plot of Figure~\ref{fig:stefan}.

\begin{figure}
  \begin{tikzpicture}
  \begin{groupplot}[group style={group size= 2 by 1},width=0.5\textwidth,cycle list/Dark2-6,
      cycle multiindex* list={
        mark list*\nextlist
      Dark2-6\nextlist},
      every axis plot/.append style={ultra thick},
      grid=major,
    ]
    \nextgroupplot[ymode=log,xmode=log,
      legend entries={\tiny{$|u-u_{h}|_{\cU}$},\tiny{$\widetilde\rho$}},
    legend pos=south west,xlabel={number of elements $\#\cPten$}]
    \addplot table [x=nEl,y=estimators] {data/ExampleStefanSmoothTen.dat};
    \addplot table [x=nEl,y=error] {data/ExampleStefanSmoothTen.dat};

    \logLogSlopeTriangle{0.8}{0.2}{0.3}{0.5}{black}{{\small $0.5$}};
    \nextgroupplot[ymode=log,xmode=log,
      legend entries={\tiny{$|u-u_{h}|_{\cU}$},\tiny{$\widetilde\rho$}},
    legend pos=south west,xlabel={number of elements $\#\cPtri$}]
    \addplot table [x=nE,y=estVol] {data/ExampleStefanSmoothTri.dat};
    \addplot table [x=nE,y=errTot] {data/ExampleStefanSmoothTri.dat};

    \logLogSlopeTriangle{0.9}{0.2}{0.2}{0.5}{black}{{\small $0.5$}};
  \end{groupplot}
\end{tikzpicture}
  \caption{Stefan problem}
  \label{fig:stefan}
\end{figure}
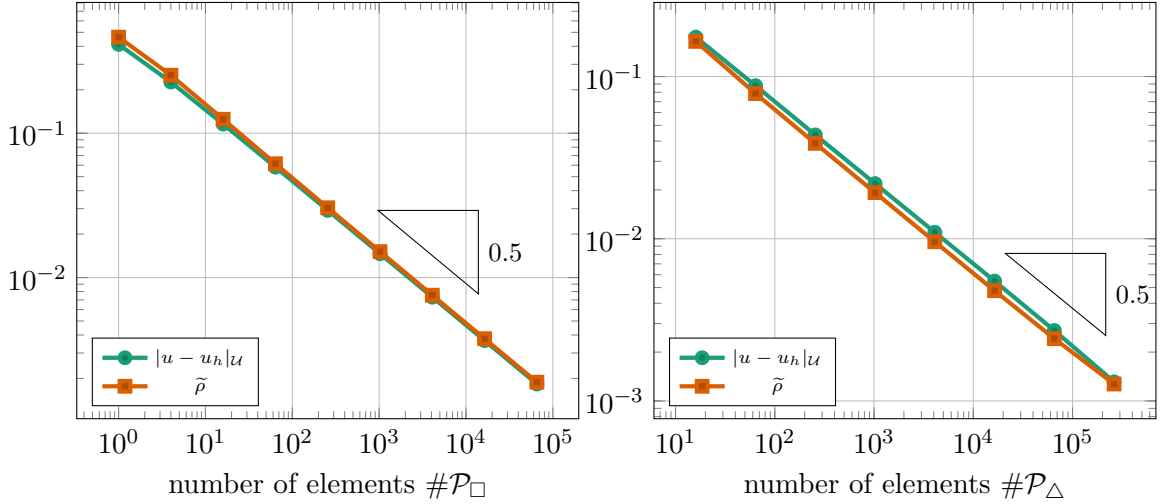

\subsection{Example with pyramid-like obstacle}\label{sec:numerics:pyramide}
For this example we consider model problem~\eqref{eq:model} with $Q = (0,1)^2$, $f = 0$, $u_0 = 0$, $\bA = \boldsymbol{I}$ (identity), $\bb = 0$, $c=0$
and the obstacle is given by
\begin{align*}
  g(t,x) = \max\{\mathrm{dist}( (t,x),\partial Q)-\tfrac14,0\}, \quad (t,x)\in Q.
\end{align*}
Here, $\mathrm{dist}( (t,x),\partial Q)$ denotes the distance of a point $(t,x)$ to the boundary of $Q$.
We note that $g\in H^{3/2-\varepsilon}(Q)\cap H_0^1(Q)$ for all $\varepsilon>0$.

\begin{figure}
  \begin{tikzpicture}
  \begin{groupplot}[group style={group size= 2 by 1},width=0.5\textwidth,cycle list/Dark2-6,
      cycle multiindex* list={
        mark list*\nextlist
      Dark2-6\nextlist},
      every axis plot/.append style={ultra thick},
      grid=major,
    ]
    \nextgroupplot[ymode=log,xmode=log,
      legend entries={\tiny{$\widetilde\rho$, unif},\tiny{$\widetilde\rho$, adap.}},
    legend pos=south west,xlabel={number of elements $\#\cPten$}]
    \addplot table [x=nEl,y=estimators] {data/ExamplePyramideUnifTen.dat};
    \addplot table [x=nEl,y=estimators] {data/ExamplePyramideAdapTen.dat};

    \logLogSlopeTriangle{0.85}{0.2}{0.5}{0.25}{black}{{\small $0.25$}};
    \logLogSlopeTriangleBelow{0.85}{0.2}{0.2}{0.5}{black}{{\small $0.5$}}
    \nextgroupplot[ymode=log,xmode=log,
      legend entries={\tiny{$\widetilde\rho$, unif},\tiny{$\widetilde\rho$, adap.}},
    legend pos=south west,xlabel={number of elements $\#\cPtri$}]
    \addplot table [x=nE,y=estVol] {data/ExamplePyramideUnifTri.dat};
    \addplot table [x=nE,y=estVol] {data/ExamplePyramideAdapTri.dat};

    \logLogSlopeTriangle{0.8}{0.2}{0.4}{0.28}{black}{{\small $0.28$}};
    \logLogSlopeTriangleBelow{0.88}{0.2}{0.1}{0.47}{black}{{\small $0.47$}}
  \end{groupplot}
\end{tikzpicture}
  \caption{Pyramide obstacle}
  \label{fig:pyramide}
\end{figure}
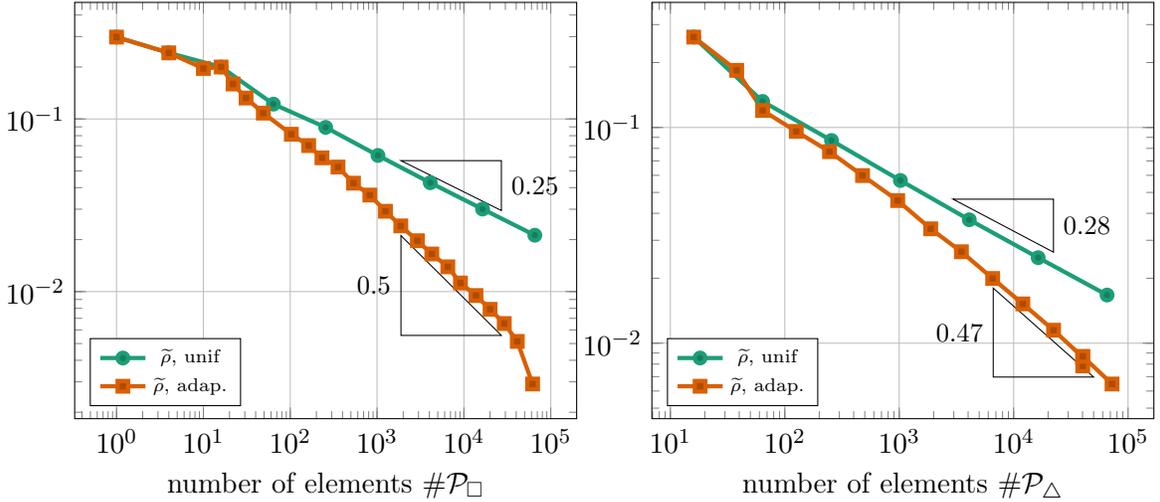

\begin{figure}
  \begin{tikzpicture}
  \begin{groupplot}[
      group style={group size=2 by 2, horizontal sep=1cm, vertical sep=2cm},
      width=0.5\textwidth,
      ylabel={$t$},
      xlabel={$x$},
    ]
    \nextgroupplot[title={\small $\#\cPtri=1884$},axis equal,hide axis]
      \addplot[patch,color=white,
        faceted color = black, line width = 0.15pt,
      patch table ={data/ExamplePyramide_ele_01884.dat}] file{data/ExamplePyramide_coo_01884.dat};
    \nextgroupplot[zmin=-0.05,zmax=0.3,colorbar]
      \addplot3[patch] table{data/ExamplePyramide_sol_01884.dat};
    \nextgroupplot[title={\small $\#\cPtri=3496$},axis equal,hide axis]
      \addplot[patch,color=white,
        faceted color = black, line width = 0.15pt,
      patch table ={data/ExamplePyramide_ele_03496.dat}] file{data/ExamplePyramide_coo_03496.dat};
    \nextgroupplot[zmin=-0.05,zmax=0.3,colorbar]
      \addplot3[patch] table{data/ExamplePyramide_sol_03496.dat};
  \end{groupplot}
\end{tikzpicture}
  \caption{Left column shows simplicial meshes generated by adaptive algorithm for the problem from Section~\ref{sec:numerics:pyramide} (vertical axis corresponds to time). Right column shows solution component $u_{\cP}$.}
  \label{fig:pyramideSol}
\end{figure}
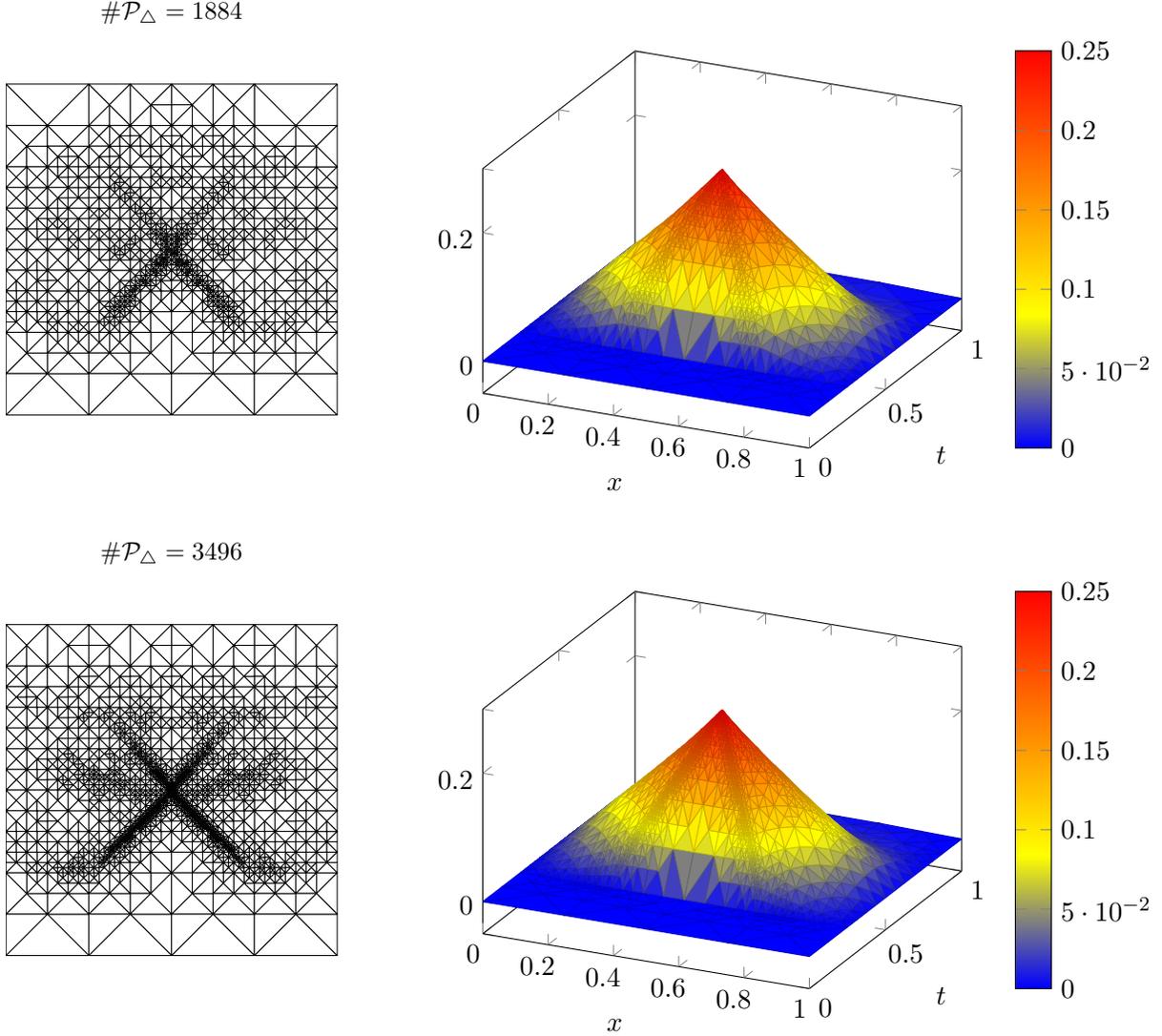

For tensor product meshes $\cPten$ and simplicial meshes $\cPtri$ we visualize convergence of the overall estimator in the left and right plot of Figure~\ref{fig:pyramide}, respectively.
For quasi-uniform meshes we observe a reduced rate for the estimator, i.e., $\widetilde\rho=\OO((\#\cPten)^{-0.25})$ resp. $\widetilde\rho = \OO((\#\cPtri)^{-0.28})$, 
whereas for a sequence of locally refined meshes generated by the adaptive algorithm with marking parameter $\theta = \frac12$ we observe about twice the rate, i.e., $\widetilde\rho = \OO( (\#\cPten)^{-0.5})$ resp. $\widetilde\rho = \OO( (\#\cPtri)^{-0.47})$.

In Figure~\ref{fig:pyramideSol} we present two meshes generated by the adaptive algorithm and corresponding solution components $u_\cP$ for $\cP=\cPtri$. 
We observe that refinements are concentrated at the lines $t=x$ and $t=1-x$ with $t<1/2$.

\subsection{American option pricing}\label{sec:blackscholes}
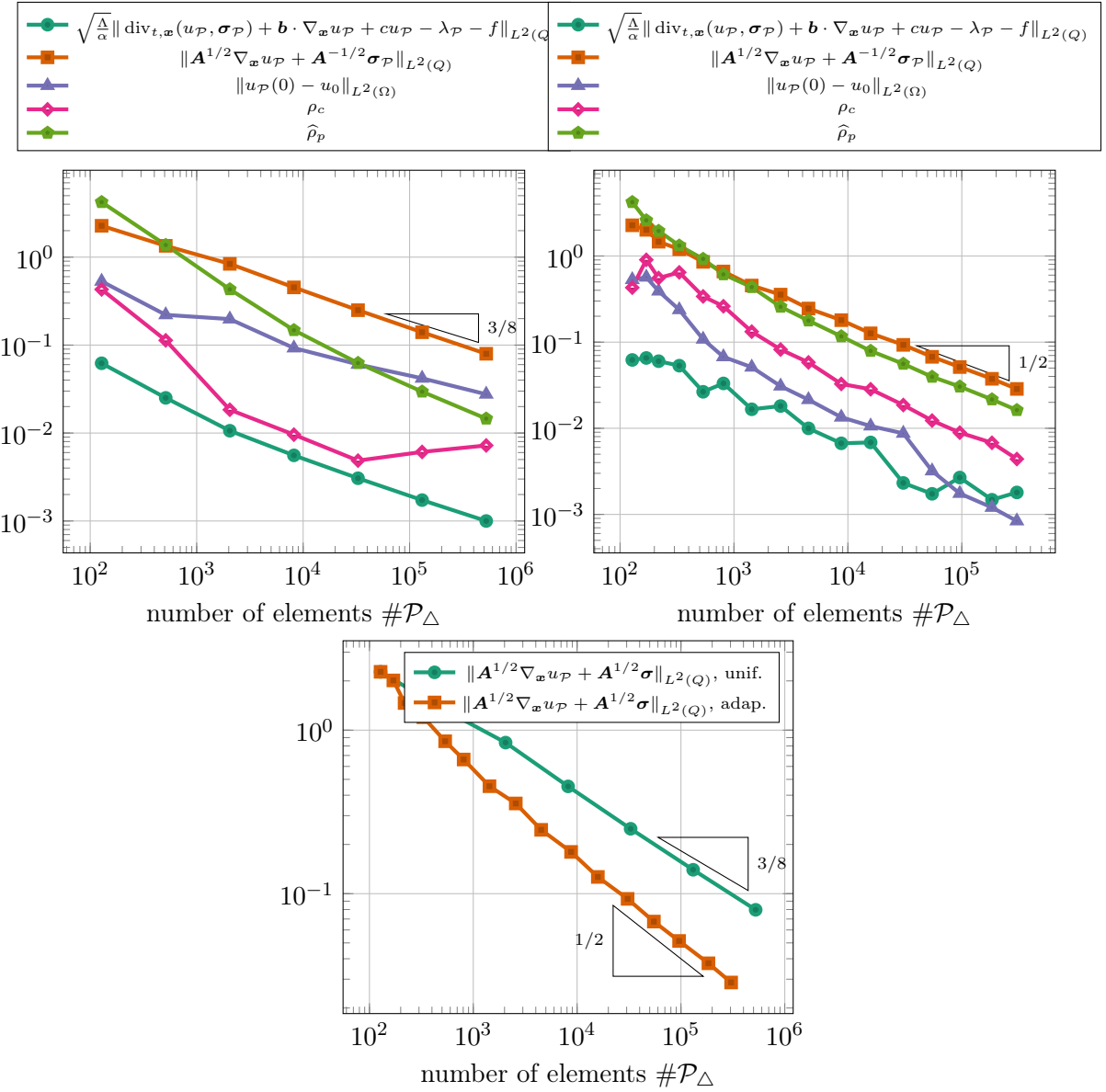
\begin{figure}
  \begin{tikzpicture}
  \begin{groupplot}[group style={group size= 2 by 1},width=0.5\textwidth,cycle list/Dark2-6,
      cycle multiindex* list={
        mark list*\nextlist
      Dark2-6\nextlist},
      every axis plot/.append style={ultra thick},
      grid=major,
      legend entries={{\tiny $\sqrt{\tfrac{\Lambda}{\alpha}}\|\divst(u_\cP,\bsigma_\cP)\-+\-\bb\cdot\gradx u_\cP+cu_\cP-\lambda_\cP-f\|_{L^2(Q)}$},{\tiny $\|\bA^{1/2}\gradx u_\cP+\bA^{-1/2}\bsigma_\cP\|_{L^2(Q)}$}, {\tiny $\|u_\cP(0)-u_0\|_{L^2(\Omega)}$}, {\tiny $\rho_c$}, {\tiny $\widehat\rho_p$}},
      legend style={at={(0.5,1.05)},anchor=south},
      legend columns=1, 
      legend style={
        /tikz/column 2/.style={
          column sep=5pt,
      }},
    ]
    \nextgroupplot[ymode=log,xmode=log,xlabel={number of elements $\#\cPtri$}
      ]
    \addplot table [x=nE,y=estDiv] {data/ExampleAOUnifTri.dat};
    \addplot table [x=nE,y=estGrad] {data/ExampleAOUnifTri.dat};
    \addplot table [x=nE,y=estU0] {data/ExampleAOUnifTri.dat};
    \addplot table [x=nE,y=estCons] {data/ExampleAOUnifTri.dat};
    \addplot table [x=nE,y=estViol] {data/ExampleAOUnifTri.dat};

    \logLogSlopeTriangle{0.9}{0.2}{0.55}{0.375}{black}{{\tiny $3/8$}};
    \nextgroupplot[ymode=log,xmode=log,xlabel={number of elements $\#\cPtri$}]
    \addplot table [x=nE,y=estDiv] {data/ExampleAOAdapTri.dat};
    \addplot table [x=nE,y=estGrad] {data/ExampleAOAdapTri.dat};
    \addplot table [x=nE,y=estU0] {data/ExampleAOAdapTri.dat};
    \addplot table [x=nE,y=estCons] {data/ExampleAOAdapTri.dat};
    \addplot table [x=nE,y=estViol] {data/ExampleAOAdapTri.dat};

    \logLogSlopeTriangle{0.9}{0.2}{0.45}{0.5}{black}{{\tiny $1/2$}};
  \end{groupplot}
\end{tikzpicture}
\begin{tikzpicture}
\begin{loglogaxis}[
width=0.49\textwidth,
cycle list/Dark2-6,
cycle multiindex* list={  mark list*\nextlist
                          Dark2-6\nextlist},
xlabel={number of elements $\#\cPtri$},
grid=major,
legend entries={{\tiny $\|\bA^{1/2}\gradx u_\cP+\bA^{1/2}\bsigma\|_{L^2(Q)}$, unif.},{\tiny $\|\bA^{1/2}\gradx u_\cP+\bA^{1/2}\bsigma\|_{L^2(Q)}$, adap.}},
legend pos=north east,
every axis plot/.append style={ultra thick},
]
    \addplot table [x=nE,y=estGrad] {data/ExampleAOUnifTri.dat};
    \addplot table [x=nE,y=estGrad] {data/ExampleAOAdapTri.dat};

    \logLogSlopeTriangle{0.9}{0.2}{0.33}{0.375}{black}{{\tiny $3/8$}};
\logLogSlopeTriangleBelow{0.8}{0.2}{0.1}{0.5}{black}{{\tiny $1/2$}}

\end{loglogaxis}
\end{tikzpicture}
  \caption{American option problem from section~\ref{sec:blackscholes} on a sequence of uniformly (top left) and adaptively (top right) refined meshes. Bottom plot compares a specific estimator contribution.}
  \label{fig:AmericanOption}
\end{figure}

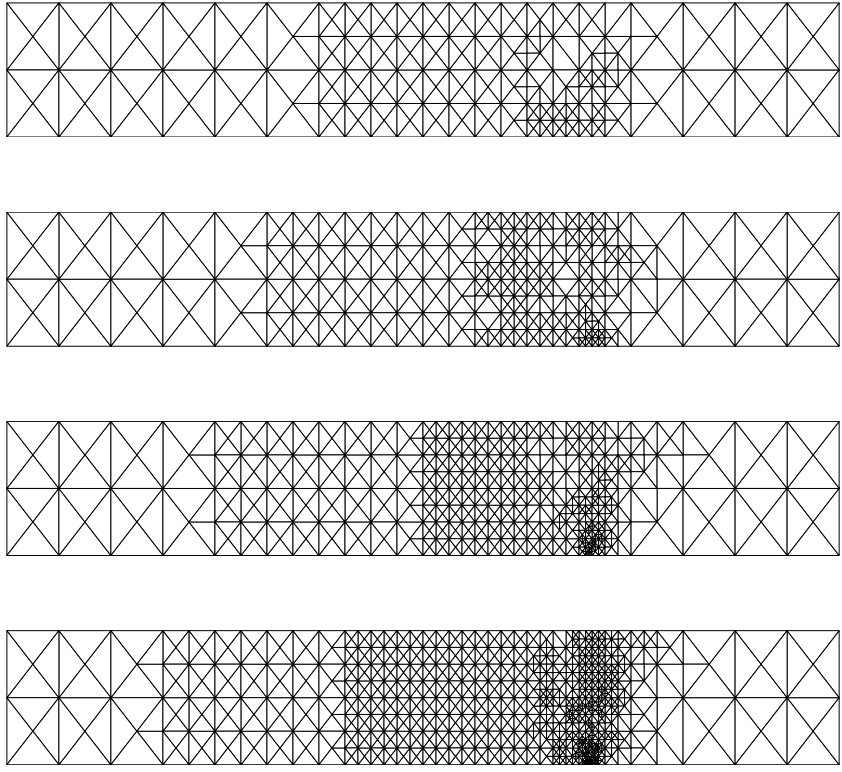
\begin{figure}
  \begin{center}
   \begin{tikzpicture}
  \begin{groupplot}[group style = {group size = 1 by 4},width=0.9\textwidth,height=0.15\textheight,ymin=0,ymax=0.5,hide axis]
    \nextgroupplot
        \addplot[patch,color=white,
          faceted color = black, line width = 0.3pt,
          patch table ={data/ExampleAO_ele_00331.dat}] file{data/ExampleAO_coo_00331.dat};
    \nextgroupplot
        \addplot[patch,color=white,
          faceted color = black, line width = 0.3pt,
          patch table ={data/ExampleAO_ele_00536.dat}] file{data/ExampleAO_coo_00536.dat};
    \nextgroupplot
        \addplot[patch,color=white,
          faceted color = black, line width = 0.3pt,
          patch table ={data/ExampleAO_ele_00808.dat}] file{data/ExampleAO_coo_00808.dat};
    \nextgroupplot
        \addplot[patch,color=white,
          faceted color = black, line width = 0.3pt,
          patch table ={data/ExampleAO_ele_01432.dat}] file{data/ExampleAO_coo_01432.dat};
  \end{groupplot}
\end{tikzpicture}
  \end{center}
  \caption{Four consecutive meshes generated by the adaptive algorithm for the American option pricing problem. Vertical axis corresponds to time.}
  \label{fig:AOmeshes}
\end{figure}
Before we present the results of this numerical experiment we describe how the American put option problem under its linear complementarity formulation falls into the framework studied in this manuscript.
For a detailed explanation and derivation of the model, we refer the reader to the book \cite{MR3677180}.

The Black--Scholes equation \cite{MR3363443} is given by
\begin{align}\label{eq:BS_eq}
  \partial_{\tau}V + \frac{\sigma^2}{2}S^2 \partial_{S}^2V + r S \partial_S V - rV = 0 \qquad \text{ in } (0,T)\times\R_+,
\end{align}
supplemented with the \emph{terminal condition} $V(T,S) = (K - S)^+$, at expiration date $T$.
In~\eqref{eq:BS_eq}, $V=V(\tau, S)$ denotes the value of the option, $S$ the current price of the underlying asset, $\sigma$ the (annual) volatility of $S$, $r\geq 0$ the (annual) interest rate, and $K$ the strike price. 
Using the time to maturity $t = T - \tau$ and $x = \log S$ as independent variables, we can rewrite \eqref{eq:BS_eq} in terms of the function $U(t,x) := V(T-\tau,e^{x})$ as follows:
\begin{align}\label{eq:BS_eq_modif}
  \partial_{t}U + \mathcal{L}_{\textrm{BS}}U:=\partial_{t}U- \frac{\sigma^2}{2} \partial_{x}^2 U - \left(r - \frac{\sigma^{2}}{2}\right) \partial_x U + r U = 0 \qquad \text{ in } (0,T)\times\R,
\end{align}
and the \emph{initial condition} $U(0,x)=(K-e^{x})^+$.
Common assumptions for this problem include a frictionless market and no arbitrage opportunities (cf. \cite[Section 1.2]{MR3677180}).
At any time $t\in (0,T)$, the option holder chooses between two possibilities: \emph{Holding} the option, in which case \eqref{eq:BS_eq_modif} applies, and \emph{exercising} the option, in which case the option value $U$ is $(K-e^{x})^+$.
Consequently, $U(t,x)$ satisfies 
\begin{align*}
  U(t,x)\geq (K-e^{x})^+ \quad \forall (t, x) \in (0,T)\times \R.
\end{align*}
The Black--Scholes equation changes to an inequality in the stopping region, i.e., $\partial_{t}U + \mathcal{L}_{\textrm{BS}}U \geq 0$ if $U(t,x) = (K-e^{x})^+$,  which leads to the complementarity condition
\begin{align*}
  (\partial_{t}U + \mathcal{L}_{\textrm{BS}}U)(U(t,x) - (K-e^{x})^+) = 0 \quad \forall (t, x) \in (0,T)\times \mathbb{R}.
\end{align*}
We note that this problem has to be solved in $\mathbb{R}$. Instead, for practical computations, one uses a bounded interval $\Omega := (L, R)$, $L<0<R$, and supplements this problem with the boundary conditions $u(t,L) = K-e^{L}$ and  $u(t,R) = 0$.
Note that this truncation introduces an error which decreases exponentially and that is negligible for a sufficiently large $R$, $|L|$, as has been also discussed in~\cite{MR2355709}.
Gathering the previous considerations, we obtain a problem within the framework of~\eqref{eq:model} with nonhomogeneous Dirichlet boundary conditions.
To recover homogeneous conditions, we introduce $\widetilde u(t,x) = (K-e^{L})^+(R-x)/(R-L)$ and consider the decomposition $U(t,x) = u(t,x) + \widetilde{u}(t,x)$.
Thus, $u$ solves model problem~\eqref{eq:model} with $f:=-\cL\widetilde u$, $g:=(K-e^x)^+-\widetilde u$, $u_0(x):=(K-e^x)^+-\widetilde u(0,x)$ and $\cL=\cL_{\mathrm{BS}}$ with coefficients $\bA = \alpha= \sigma^2/2$, $\bb =-(r-\alpha)$, $c=r$. 
  Clearly, since all coefficients are constant, $-\tfrac12\divx\bb+c = c > 0$.

We take the example from~\cite[Section~5.4]{MR2355709}
with $\sigma = 0.4$, $r=0.06$, $R = 7$, $L=-1$, and $K=100$.
From the results of numerical experiments (not shown here) we found that estimating the dual norm in the estimator $\rho_p$ by the $L^2(Q)$ norm to define the estimator $\widetilde\rho_p$, see Remark~\ref{rem:estimator}, is too rough, in the sense that $\widetilde\rho_p$ would dominate the other error contributions.
Instead, for this example we define $\widehat\rho_p$ as $\rho_p$ where we only replace $\|\partial_t(g-u_\cP)^+\|_{\cV^*}$ by the weighted norm $\|h_\cP\partial_t(g-u_\cP)^+\|_{L^2(Q)}$. 
Although, $\partial_t(g-u_\cP)^+$ is in general not a polynomial function, the replacement is motivated by the fact that $(g-u_\cP)^+$ vanishes at least at the vertices of the triangulation (where the constraints of the set $\cKsymDisc$ are enforced).

Figure~\ref{fig:AmericanOption} shows in the first row the error estimator contributions in case of uniformly and adaptively refined simplicial meshes, respectively.
We observe that $\|\bA^{1/2}u_\cP+\bA^{-1/2}\bsigma_\cP\|_{L^2(Q)}$ is the dominating contribution in the error estimator and therefore dictates the convergence rate. 
In the case of uniform refinements we obtain approximately $\OO((\#\cPtri)^{-3/8})$ convergence, which corresponds to $\OO(h^{3/4})$ with $h$ denoting the uniform mesh-size.
This has been also observed in~\cite[Sec.~5.4]{MR2355709} for the time-stepping method and for the estimator defined there (though they do not include the dual norm in $\rho_p$ or a similar term as we do in the work at hand).
From the right plot in the first row we see that $\OO((\#\cPtri)^{-1/2})$ convergence is recovered when employing the adaptive algorithm.
Figure~\ref{fig:AOmeshes} shows meshes produced by the adaptive algorithm. We observe strong refinements towards the region around $t=0$ and $x=\ln K$. We note that the obstacle function $g$ has a kink along the line $x = \ln K \approx 4.6$.

\subsection{Heat equation obstacle problem for $d=2$}\label{sec:numerics:2d}
\begin{figure}
  \begin{tikzpicture}
  \begin{groupplot}[group style={group size= 2 by 1},width=0.5\textwidth,cycle list/Dark2-6,
      cycle multiindex* list={
        mark list*\nextlist
      Dark2-6\nextlist},
      every axis plot/.append style={ultra thick},
      grid=major,
    ]
    \nextgroupplot[ymode=log,xmode=log,legend entries={{\tiny $|u-u_\cP|_{\cU}$}, {\tiny $\widetilde\rho_c$}},
      legend pos=south west,xlabel={number of elements $\#\cPtri$}
      ]
    \addplot table [x=nE,y=errUnorm] {data/Example2D.dat};
    \addplot table [x=nE,y=est] {data/Example2D.dat};

    \logLogSlopeTriangle{0.9}{0.2}{0.2}{0.333}{black}{{\tiny $1/3$}};
    \nextgroupplot[ymode=log,xmode=log,legend entries={{\tiny $\|\divst(u_\cP,\bsigma_\cP)-\lambda_{\cP}-f\|_{L^2(Q)}$}, {\tiny $\|\gradx u_\cP+\bsigma_\cP\|_{L^2(Q)}$},{\tiny $\|u_\cP(0)-u_0\|_{L^2(\Omega)}$},{\tiny $\widetilde\rho_p$},{\tiny $\rho_c$}},
      legend pos=south west,xlabel={number of elements $\#\cPtri$}]
    \addplot table [x=nE,y=estDiv] {data/Example2D.dat};
    \addplot table [x=nE,y=estGrad] {data/Example2D.dat};
    \addplot table [x=nE,y=errU0] {data/Example2D.dat};
    \addplot table [x=nE,y=estPen] {data/Example2D.dat};
    \addplot table [x=nE,y=estViol] {data/Example2D.dat};

    \logLogSlopeTriangle{0.9}{0.2}{0.68}{0.333}{black}{{\tiny $1/3$}};
  \end{groupplot}
\end{tikzpicture}
  \caption{Error and estimator (contributions) for the problem described in Section~\ref{sec:numerics:2d}.}
  \label{fig:Example2D}
\end{figure}
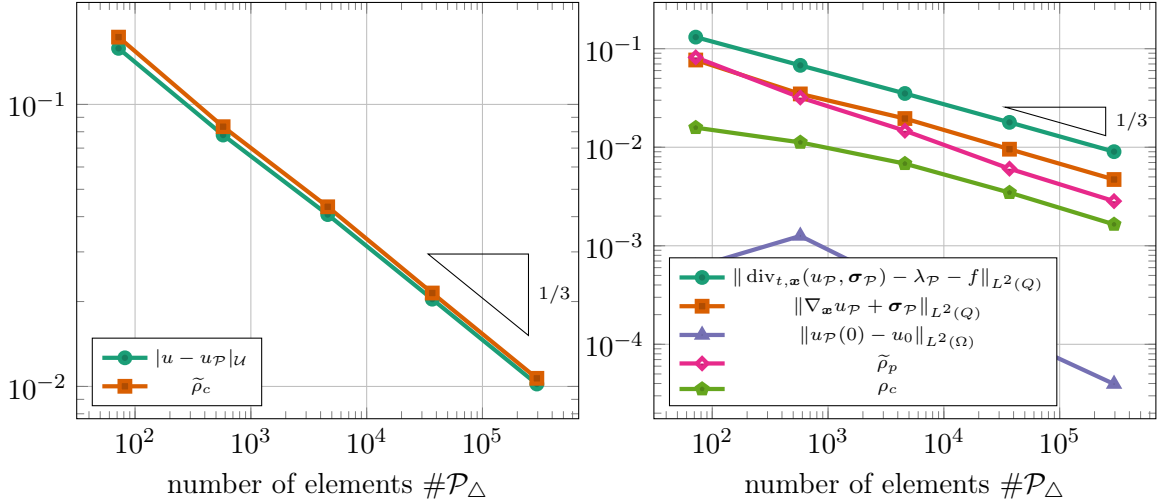
In our final experiment we consider the computational domain $Q = (0,1)^3$ and use a tetrahedral mesh. 
We use data
\begin{align*}
  u_0(x,y) &= 0, \\
  f(t,x,y) &= \begin{cases}
    2x\big((1-x)x(1-y)y + 2t(x(1-x)+y(1-y))\big) & x< \tfrac12, \\
    (1-x)x(1-y)y + 2t(x(1-x)+y(1-y)) & x\geq \tfrac12,
  \end{cases}\\
  g(t,x,y) &= \begin{cases}
    t(1-x)x(1-y)y & x<\tfrac12, \\
    t\widetilde g(x) (1-y)y & x\in[\tfrac12,\tfrac34],\\
    0 & x>\tfrac34
  \end{cases}
\end{align*}
in problem~\eqref{eq:model} with $\cL = -\Delta_\bx$.
Here, $\widetilde g$ is the unique polynomial of degree $3$ such that $g$ and its derivatives are continuous at $x=\tfrac12,\tfrac34$, see also~\cite[Section~5.3]{LSQobstacle} for a similar example for the elliptic obstacle problem. 
Our choice leads to the exact solution
\begin{align*}
  u(t,x,y) = t(1-x)x(1-y)y, \quad (t,x,y)\in Q.
\end{align*}
Note that $u$ is smooth. Furthermore, $f\in H^1(Q)$, $g\in C^\infty([0,T];H^{5/3-\varepsilon}(\Omega))$ for any $\varepsilon>0$. 
In particular, we expect the optimal convergence $|\bu-\bu_\cP|_\cU = \OO(h)$ with $h\eqsim (\#\cP)^{-1/3}$ denoting the uniform mesh-size.
This is observed in Figure~\ref{fig:Example2D} where on the left we compare the error and estimator $\widetilde\rho$ on a sequence of uniformly refined meshes and on the right we compare the different estimator contributions. 

\bibliographystyle{alpha}
\bibliography{literature}

\end{document}